\documentclass[reqno,english]{amsart}%
\usepackage{amsfonts,amsmath,latexsym,verbatim,amscd,mathrsfs,color,array}
\usepackage[colorlinks=true]{hyperref}
\usepackage{amsmath,amssymb,amsthm,amsfonts,graphicx,color}
\usepackage{amssymb}
\usepackage{pdfsync}
\usepackage{epstopdf}
\usepackage{cite}
\usepackage{graphicx}
\usepackage{amsmath}
\usepackage{amsfonts}%
\usepackage{cases}
\usepackage[a4paper]{geometry}
\providecommand{\U}[1]{\protect\rule{.1in}{.1in}}

\numberwithin{equation}{section}

\newtheorem{theorem}{Theorem}[section]
\newtheorem{corollary}{Corollary}[section]
\newtheorem{lemma}{Lemma}[section]
\newtheorem{proposition}{Proposition}[section]
\newtheorem{remark}{Remark}[section]
\newtheorem{definition}{Definition}[section]

\numberwithin{equation}{section}

\newcommand{\into}{\int_\Omega}
\newcommand{\bbr}{\mathbb{R}}

\newcommand{\ve}{\varepsilon}
\newcommand{\W}{\mathcal{W}}

\newcommand{\Wj}{\mathcal{W}_{\mu_j,\xi_j}}

\newcommand{\Wh}{\mathcal{W}_{\mu_h,\xi_h}}
\newcommand{\uu}{\mathcal{U}}

\newcommand{\bd}{\begin{definition}}
\newcommand{\ed}{\end{definition}}

\newcommand{\br}{\begin{remark}}
\newcommand{\er}{\end{remark}}

\newcommand{\be}{\begin{equation}}
\newcommand{\ee}{\end{equation}}

\newcommand{\bc}{\begin{corollary}}
\newcommand{\ec}{\end{corollary}}

\begin{document}

\title[Brezis-Nirenberg problem]{On Brezis-Nirenberg problems: open questions and new results in dimension six}

\author[F. Li]{Fengliu Li}
\address{\noindent  School of Mathematics, China
University of Mining and Technology, Xuzhou, 221116, P.R. China }
\email{lifengliu@cumt.edu.cn}

\author[G. Vaira]{Giusi Vaira}
\address{\noindent  Dipartimento di Matematica, Universit\'a degli Studi di Bari Aldo Moro,Italy }
\email{giusi.vaira@uniba.it}

\author[J. Wei]{Juncheng Wei}
\address{\noindent Department of Mathematics, Chinese University of Hong Kong,
Shatin, NT, Hong Kong}
\email{wei@math.cuhk.edu.hk}

\author[Y. Wu]{Yuanze Wu}
\address{\noindent  School of Mathematics, Yunnan Key Laboratory of Modern Analytical Mathematics and Applications, Yunnan Normal University, Kunming, 650500, P.R. China}
\email{yuanze.wu@ynnu.edu.cn}

\thanks{G. Vaira is partially supported by
the MUR-PRIN-P2022YFAJH ``Linear and Nonlinear PDE's: New directions and Applications" and by the INdAM-GNAMPA project ``Fenomeni non lineari: problemi locali e non locali e loro applicazioni",  CUP E5324001950001. J. Wei is partially supported by GRF of Hong Kong entitled "New frontiers in singularity formations in nonlinear partial differential equations".   Y. Wu is supported by NSFC (No. 12171470). F. Li is partially supported by NSFC (No. 12171470) and the Graduate Innovation Program of China University of Mining and Technology 2025WLKXJ133}

\date{}

\maketitle
{\centering  \emph{To the memory of Louis Nirenberg with admiration}\par}
\begin{abstract}
In this paper, we consider the Brezis-Nirenberg problem
\begin{equation*}
\left\{\begin{aligned}
&-\Delta u = \lambda u+|u|^{2^*-2}u, \quad &\mbox{in}\,\Omega,\\
&u=0,\quad &\mbox{on}\, \partial\Omega,
\end{aligned}\right.
\end{equation*}
where $\Omega $ is a smoothly bounded domain of $\mathbb R^N$ with $N\geq 3$, $\lambda>0$ is a parameter and $2^*=\frac{2N}{N-2}$ is the critical Sobolev exponent.  We first recall the history of the Brezis-Nirenberg problem and then provide new results of it in dimension six.  Finally, we also list some open questions on the Brezis-Nirenberg problem.

\vspace{3mm} \noindent{\bf Keywords:} Brezis-Nirenberg problem; Multi-bump solution; Bubbling phenomenon; Sign-changing solution.

\vspace{3mm}\noindent {\bf AMS} Subject Classification 2020: 35B33; 35B40; 35B44; 35J15.%
\end{abstract}

\section{Introduction}

In this paper we consider the Brezis-Nirenberg problem, namely we let
\begin{equation}\label{pb01}
\left\{\begin{aligned}
&-\Delta u = \lambda u+|u|^{2^*-2}u, \quad &\mbox{in}\,\Omega,\\
&u=0,\quad &\mbox{on}\, \partial\Omega,
\end{aligned}\right.
\end{equation}
where $\Omega $ is a smoothly bounded domain of $\mathbb R^N$ with $N\geq 3$, $\lambda>0$ is a parameter and $2^*=\frac{2N}{N-2}$ is the critical Sobolev exponent.

\vskip0.12in

\subsection{History of \eqref{pb01}}
Problem~\eqref{pb01} is introduced by Brezis and Nirenberg in their celebrated paper \cite{BN}, where they gave a detailed study on the existence of a positive solution of it.   By applying the Pohozaev identity \cite{Po}, Brezis and Nirenberg first pointed out in \cite{BN} that problem \eqref{pb01} only has the trivial solution when $\Omega$ is star-shaped and $\lambda\leq 0$, and by multiplying \eqref{pb01} with the unique eigenfunction of $\lambda_1(\Omega)$, where $\lambda_1(\Omega)$ is the first eigenvalue of $-\Delta$ with Dirichlet boundary condition, Brezis and Nirenberg next pointed out in \cite{BN} that problem \eqref{pb01} has no positive solutions for $\lambda\geq\lambda_1(\Omega)$.  Moreover, they also showed in \cite{BN} that for $N\geq 4$ the problem~\eqref{pb01} admits a positive solution for every $\lambda\in (0, \lambda_1(\Omega))$.  The situation in dimension $N=3$ is delicate and in \cite{BN} Brezis and Nirenberg are able to prove the existence of $\lambda_*(\Omega)>0$ such that a positive solution of \eqref{pb01} exists for every $\lambda\in (\lambda^*(\Omega), \lambda_1(\Omega))$.  The existence of the positive value $\lambda_*(\Omega)$ in the existence theory of positive solutions of \eqref{pb01} is necessary in general bounded domains in dimension three, since they also showed that if $\Omega\subset \mathbb R^3$ is strictly star-shaped about the origin, then there are no positive solutions of \eqref{pb01} for $\lambda$ close to zero.  When $\Omega$ is the unit ball $B_1$ in dimension three, the existence theory of positive solutions of \eqref{pb01} and the value of $\lambda_*(B_1)$ are completely known, i.e. $\lambda_*(B_1)=\frac{\lambda_1(B_1)}{4}$ and a positive solution of \eqref{pb01} exists if and only if $\lambda\in \left(\frac{\lambda_1(B_1)}{4}, \lambda_1(B_1)\right)$.  Conversely, Bahri and Coron \cite{BC} presented an existence result for a positive solution to problem \eqref{pb01} for $\Omega$ with a nontrivial topology and $\lambda=0$.

\vskip0.12in

In the last decades problem \eqref{pb01} has attracted the attention of many researchers due to the critical behavior and the lack of compactness of the Brezis-Nirenberg problem~\eqref{pb01} that prevents to use standard variational methods.  A first question is about the asymptotic behavior of the solution founded by Brezis and Nirenberg. This was studied by Han \cite{H} and Rey \cite{R1} and they showed, with different proofs, that such solutions blow up at exactly one point. They also determined the exact blow-up rate as well as the location of the limit concentration point which is a critical point of the Robin function (see also \cite{W} in which it was shown that this concentration point is a global minimum of the Robin function). Bahri, Li and Rey \cite{BLO} determined the asymptotic behavior and locations of multi-bubbling solutions to Brezis-Nirenberg problem.  Most recently, a fine multi-bubbles analysis of positive solutions has been object of two papers by K\"onig and Laurin. Indeed, in \cite{KL2} the authors give an accurate description of the blow-up profile of a solution with multiple blow-up points when $N\geq 4$ while in \cite{KL1} they consider the case $N = 3$. The 3-dimensional case was firstly faced by Druet \cite{D}.
In particular, the asymptotic analysis of blowing-up solutions ensures that the blow-up points are nothing but the critical points of a suitable function which involves the Green function of $-\Delta$ in $\Omega$ with Dirichlet boundary condition and the Robin function.  Furthermore, a counterpart of these asymptotic results are studied.  The first result in this direction is due to Rey that built in \cite{R} a solution that blows up at one non-degenerate critical point of the Robin function when $N\geq 5$ and $\lambda\to 0^+$.  In \cite{MP}, Musso and Pistoia constructed positive solutions with multiple blow-up points when $N\geq 5$ as $\lambda\to 0^+$ whereas, Musso and Salazar \cite{MS} considered the case of multiple blow-up points in dimension $N=3$ as $\lambda$ goes to the strictly positive number $\lambda^*$ predicted by Druet in \cite{D} (see also \cite{DDM} for the construction of one-peak solution as $\lambda\to \lambda^*$ in dimension $N=3$).  It is only very recently that the first true bubbling solutions have been constructed in dimension $N=4$ as $\lambda\to 0^+$, see \cite{PR} for the case of one peak solution and \cite{PRV} for the case of multiple peak solutions.

\vskip0.12in

For what concern sign-changing solutions of \eqref{pb01}, several existence results have been obtained for $N\geq4$. Indeed one can get sign-changing solutions for every $\lambda>0$, $\lambda\not\in\sigma(-\Delta)$ if $N=4$ and for every $\lambda>0$ if $N\geq 5$ (see \cite{CFP,SWW}).  The case $N=3$ presents even more difficulties  (as for positive solutions) and the question on what is the range of $\lambda$ for which a sign-changing solution of \eqref{pb01} exists in dimension $N=3$ has been raised by Brezis in a collection of his favorite open problem \cite{B}, which, as said by Brezis, is open for 40 years.  We point out that, very recently, by using the del Pino-Musso-Pacard-Pistoia sign-changing solutions of the Yamabe problem in \cite{DMPP2011} as the building blocks, Sun, Wei and Yang constructed infinitely many non-radial sign-changing solutions of the 3-dimensional Brezis-Nirenberg problem~\eqref{pb01} in $\Omega=B_1$ the unit ball in $\mathbb{R}^3$ and for any $\lambda>0$ in \cite{SWY}, which gives a complete answer to the Brezis' first open problem proposed in \cite{B}. 

\vskip0.12in

As for the multiplicity results, which were obtained mainly through non-standard variational techniques, the first contribution was due to Cerami, Fortunato, and Struwe in \cite{CFS}, in which they showed that the number of nontrivial solutions of \eqref{pb01}, for $N\geq 3$, is bounded below by the number of eigenvalues of $-\Delta$ belonging to $(\lambda, \lambda+S|\Omega|^{-\frac 2 N})$, where $S$ is the best constant for the Sobolev embedding and $|\Omega|$ is the Lebesgue measure of $\Omega$.  Moreover, in \cite{FJ} it was also proved that if $N\geq 4$, then for any $\lambda>0$ and for a suitable class of symmetric domain $\Omega$, problem \eqref{pb01} has infinitely many solutions of arbitrarily large energy and if $N\geq 7$ and $\Omega$ is a ball, then for each $\lambda>0$, problem \eqref{pb01} has infinitely many sign-changing radial solutions (see \cite{S}). Here the radial symmetry of the domain plays an essential role, therefore their methods do not work for general domains.  Furthermore, Cerami, Solimini and Struwe showed in \cite{CSS} that if $N\geq 6$ and $\lambda\in (0, \lambda_1(\Omega))$, problem \eqref{pb01} has a pair of least energy sign-changing solution. In the same paper the authors studied the multiplicity of nodal solutions proving the existence of infinitely many radial solutions when $\Omega$ is a ball centered at the origin. Moreover, Devillanova and Solimini in \cite{DS1} showed that, if $N\geq 7$ and $\Omega$ is an open regular subset of $\mathbb R^N$, problem \eqref{pb01} has infinitely many solutions for each $\lambda>0$.  For low dimensions, namely $N=4, 5, 6$ and in an open regular subset of $\mathbb R^N$ , in \cite{DS2}, Devillanova and Solimini proved the existence of at least $N+1$ pairs of solutions provided $\lambda$ is small enough.  This result was extended to all $\lambda > 0$  in \cite{CW} by Clapp and Weth.
Neither in \cite{DS1, DS2} nor in \cite{CW} there is an information on the kind of solutions obtained.  Recently, in \cite{SZ}, Schechter and Wenming Zou showed that in any bounded and smooth domain, for $N\geq 7$ and for each fixed $\lambda>0$, problem \eqref{pb01} has infinitely many sign changing solutions.

\vskip0.12in

Concerning the profile of sign-changing solutions some results have been obtained in \cite{BEP1, BEP2} for low energy solutions, namely solutions $u_\lambda$ such that
$\|u_\lambda\|_{H^1_0}^2\to 2S^{\frac N 2}$ as $\lambda\to 0$. More precisely in \cite{BEP1} it was considered the case $N=3$ and they proved
that these solutions concentrate and blow-up in two different points of $\Omega$, as $\lambda\to 0$, and have the asymptotic profile of two separate bubbles. A similar result is proved in \cite{BEP2} for $N\geq 4$ but assuming that the blow-up rate of the positive and negative part of $u_\lambda$ is the same.  This means that, for $N\geq 4$, a solution whose positive part and negative part blow-up at the same point with different speeds of concentration can exist. Indeed, by analyzing the asymptotic behavior of the least energy nodal radial solutions $u_\lambda$ in the ball, as $\lambda\to 0$ in dimension $N\geq 7$ one can prove (see \cite{Iac}) that their limit profile is that of a
{\it tower of two bubbles}. In \cite{IV1} it was shown that, in some symmetric bounded domains, a sign-changing bubble tower exists for $\lambda\to 0$ and $N\geq 7$ and moreover in dimensions $N=4, 5, 6$ this kind of solutions do not exist (see \cite{IP}).  It is worth pointing out that existence of nodal solutions with two nodal regions concentrating in two different points of the domain $\Omega$ as $\lambda\to 0$ has been obtained in \cite{CC}, \cite{MiPi} and \cite{BMP} while other king of sign-changing solutions for $N\geq 7$ that has a concentration on a boundary point (here the domain has a non-trivial topology since the boundary concentration is not always allowed) were studied in \cite{V} and \cite{MRV}.

\vskip0.12in

From the previous discussions it follows that the critical behavior and the lack of compactness of the Brezis-Nirenberg problem~\eqref{pb01} depend strongly on the geometry of the domain and on the dimensions, and in the physical ($N=3$) and low ($4\leq N\leq6$) dimensions, there are lots of ``nonstandard'' profiles of sign-changing solutions of the Brezis-Nirenberg problem~\eqref{pb01}.  Atkinson, Brezis and Peletier \cite{ABP}  (see also \cite{AY} for a different proof) showed that for $3\leq N\leq 6$ and when $\Omega$ is a ball, there is a $\lambda_0>0$ such that \eqref{pb01} has no sign-changing radial solutions for $\lambda\in (0, \lambda_0)$.
Moreover it was also studied the asymptotic behavior of the values $\lambda$ for $3\leq N\leq 6$ (see also \cite{BP}).  By letting $\lambda_m$ the eigenvalue of $-\Delta$ on $B_1$ which corresponds to a radial eigenfunctions with $m-1$ zeros and letting $u_\lambda^m$ be a solution of \eqref{pb01} with $m-1$ zeros, they showed that
\begin{itemize}
\item if $N=3$ then for every  $m\geq 2$ we have that $\lambda\to \left(\frac{2m-1}{2}\pi\right)^2$ as $\|u_\lambda^m\|_\infty\to\infty$;
\item if $N=4, 5$ then for every $m\geq 2$ we have that $\lambda\to\lambda_{m-1}$ as $\|u_\lambda^m\|_\infty\to\infty$;
\item if $N=6$ then for every $m\geq 2$ we have $\lambda\to\lambda_{m-1}^*$ as $\|u_\lambda^m\|_\infty\to\infty$  where $\lambda^*_{m-1}\in (0, \lambda_{m-1})$.
\end{itemize}
Moreover if $m=2$ (so $u_\lambda$ has only two nodal regions) Gazzola and Grunau \cite{GG} showed that $\lambda\to\lambda^+_1$ if $N=4$ while $\lambda\to\lambda_1^-$ if $N=5$.  These analysis show that concentration values which is different from zero for $N\geq4$ and different from $\lambda^*$ for $N=3$ predicted by Druet in \cite{D} can appear, at least for radial solutions in the physical ($N=3$) and low ($4\leq N\leq6$) dimensions and thus, it is reasonable to first understand what is the asymptotic behavior of nodal radial solutions in the physical ($N=3$) and low ($4\leq N\leq6$) dimensions as $\lambda$ converges to the limit value which is different from zero for $N\geq4$ and different from $\lambda^*$ for $N=3$.

\vskip0.12in

In \cite{IP1}, the authors showed that given $u_\lambda$, a nodal radial solutions of \eqref{pb01} having two nodal regions and such that $u_\lambda(0)>0$ then
\begin{itemize}
\item[(i)] if $N = 6$ then $\lambda_2^*\in (0, \lambda_1(B_1))$ and, as $\lambda\to\lambda_2^*$, the positive part of $u_\lambda$ concentrates at the center of the ball, $\|u_\lambda^+\|_\infty\to\infty$, and a suitable rescaling of $u_\lambda^+$ converges to the standard positive solution of the critical problem in $\mathbb R^N$. Instead the negative part $u_\lambda^-$ converges to the unique positive solution of \eqref{pb01} in $B_1$, as $\lambda\to\lambda_2^*$;
\item[(ii)] if $N = 4,5$ then $u_\lambda^+$ behaves as for the case $N = 6$, while $u_\lambda^-$ converges to zero uniformly in $B_1$ as $\lambda\to\lambda_1(B_1)$;
\item[(iii)] if $N = 3$ then $\lambda_2^*= \frac 9 4 \lambda_1(B_1)$ and $u_\lambda^+$ behaves as for the case $N = 6$, while $u_\lambda^-$ converges
to zero uniformly in $B_1$.
\end{itemize}
In a recent paper \cite{AGGPV}, Amadori et al. gave a completed description of concentration values and bubbling phenomenon of the Brezis-Nirenberg problem \eqref{pb01} in the radial setting for all $N\geq 3$, which refine the results in \cite{AY, ABP, BP}.

\vskip0.12in

The existence of a solution predicted by these asymptotic analysis in \cite{AY, ABP, AGGPV, BP, IP1} was first studied in \cite{IV} where the authors considered the case $N=4, 5$ showing that in general (symmetric) bounded domains of $\mathbb R^N$ ($N=4, 5$), there exist a sign-changing solutions of problem \eqref{pb01} having, as $\lambda\to\lambda_1^\pm(\Omega)$, the following asymptotic profile $$u_\lambda\sim \mathcal W_{\mu, 0}-\tau e_1$$ where $\mathcal W_{\mu, 0}$ is the projection of the standard Aubin-Talenti bubble in $H^1_0(\Omega)$, $\mu, \tau\approx0$ and $e_1$ is the first eigenfunction of $-\Delta$ with Dirichlet boundary condition.  Very recently, in \cite{LVWW} it was generalized the result in \cite{IV}.  Indeed, by using the Lyapunov-Schmidt reduction arguments, we construct multi-bump bubbling solutions of \eqref{pb01}, which look like
\begin{eqnarray*}
u_\lambda\approx\sum_{j=1}^{k}\beta_j\mathcal{W}_{\mu_j,\xi_j}+\sum_{i=1}^{m}\tau_ie_i
\end{eqnarray*}
as $\lambda$ is close to any fixed eigenvalue of the Laplacian operator $-\Delta$ with Dirichlet boundary condition, where $\{e_i\}$ is an orthogonal system related to the fixed eigenvalue, $\tau_l\approx0$ for all $1\leq l\leq m$, $\beta_j=\pm1$, $\mathcal{W}_{\mu_j,\xi_j}$ is the projection of the standard Aubin-Talenti bubble in $H^1_0(\Omega)$, $\mu_j\approx0$ and $\xi_j\to\xi_{j,0}\in\Omega$ for all $1\leq j\leq k$.  Here, the governing functions of the bubbling phenomenon of $4D$ and $5D$ Brezis-Nirenberg problem as $\lambda$ approaches any eigenvalue of the Laplacian operator $-\Delta$ with Dirichlet boundary condition are also found by studying some variational problems associated to it.

\subsection{New results in dimension six}
The existence of a solution predicted by these asymptotic analysis in \cite{AY, ABP, AGGPV, BP, IP1} for dimension $N=6$ is more delicate, since it is known in  \cite{AY, ABP, AGGPV, BP, IP1} for the ODE setting that the profile of the bubbling solution of the $6D$ Brezis-Nirenberg problem should be very different from that of the $4D$ and $5D$ cases.  Besides, the computations in \cite{IV,LVWW} imply that the zero weak limit and the blow-down part of the ansatz can not balance the bubbles near the blow-up points, and the blow-up analysis in \cite{D2003} reveals a competitive mechanism between contributions of nonzero weak limit and linear term in dimension $N=6$.  Thus, a nonzero weak limit is needed in constructing a ture blow-up solution of the $6D$ Brezis-Nirenberg problem~\eqref{pb}, which reads as
\begin{eqnarray}\label{pb}
\left\{
\aligned
&-\Delta u=\lambda u+|u|u,\quad\text{in }\Omega,\\
&u=0\quad\text{on }\partial\Omega.
\endaligned
\right.
\end{eqnarray}
Pistoia and Vaira successfully proved in \cite{PV2021} that \eqref{pb} has a bubbling solution, which looks like
\begin{eqnarray*}
u_{\lambda}\approx u_{\lambda_0}-\mathcal{W}_{\mu,\xi}+(\lambda-\lambda_0) v_{\lambda_0}
\end{eqnarray*}
as $\lambda\to\lambda_0$ in generic bounded domains in the sense that any solution of \eqref{pb} is non-degenerate, where $\mu\approx0$, $u_{\lambda_0}$ is a ground-state solution of \eqref{pb} satisfying
\begin{eqnarray}\label{lambda0}
\lambda_0=2\max_{\Omega}u_{\lambda_0}(x)=u_{\lambda_0}(\xi_0)
\end{eqnarray}
and $v_{\lambda_0}$ is the unique solution of the following equation:
\begin{eqnarray*}
\left\{
\aligned
&-\Delta v=\left(\lambda_0+2|u_{\lambda_0}|\right)v+u_{\lambda_0},\quad&\text{in }\Omega,\\
&v=0\quad&\text{on }\partial\Omega.
\endaligned
\right.
\end{eqnarray*}
To complete the construction, it is assumed in \cite{PV2021} that $2v_{\lambda_0}(\xi_0)\neq 1$, which is naturally satisfied in the ODE setting (cf. \cite{AGGPV}).  We remark that, to our best knowledge, the introduction of nonzero weak limits in constructing bubbling solution of critical elliptic equtions first appeared in \cite{V} and can be also found in \cite{GLW2013,RV,MRV} and the references therein.

\vskip0.12in

In this paper, we shall improve the results in \cite{PV2021} in two fonts by constructing multi-bubble solutions without the assumption $2v_{\lambda_0}(\xi_0)\neq 1$ for concentration points $\xi_0$.  In order to state our main result, we first introduce necessary notations and assumptions.  By \cite[Theorem~1.2]{PV2021}, we know that for generic domains $\Omega$ with $\partial\Omega\in C^{3,\alpha}$ and $\alpha\in(0, 1)$, if $u_{\lambda}\in H^1_0(\Omega)$ solves \eqref{pb} for $\lambda>0$, then $u_{\lambda}$ is non-degenerate in $H^1_0(\Omega)$ in the sense that the Kernel of the operator $-\Delta-\left(\lambda+(2^*-1)|u_{\lambda}|^{2^*-2}\right)$, which is denoted by $\mathcal{K}_{\lambda}$, is empty in $H^1_0(\Omega)$.  Thus, we say that $\Omega$, with $\partial\Omega\in C^{3,\alpha}$ and $\alpha\in(0, 1)$, is {\bf a non-degenerate domain} of the Brezis-Nirenberg problem~\eqref{pb} if any solution of \eqref{pb} in this bounded domain $\Omega$ is non-degenerate.

\vskip0.12in

Let $u_{\lambda}$ be a solution of the $6D$ Brezis-Nirenberg problem~\eqref{pb} for $\lambda>0$ in a non-degenerate domain $\Omega$, then by the Fredholm alternative, the following equation
\begin{eqnarray}\label{eqnPreWu0026}
\left\{
\aligned
&-\Delta v=\left(\lambda+2|u_{\lambda}|\right)v+u_{\lambda},\quad&\text{in }\Omega,\\
&v=0\quad&\text{on }\partial\Omega,
\endaligned
\right.
\end{eqnarray}
has a unique solution, which is denoted by $v_{\lambda}$.  We introduce the sets
\begin{eqnarray*}
\Omega_{u_{\lambda},\pm\frac{\lambda}{2}}=\left\{x\in\Omega\mid u_{\lambda}(x)=\pm\frac{\lambda}{2}\text{ and }\nabla u_{\lambda}(x)=0\right\}.
\end{eqnarray*}
For any $\xi\in\Omega_{u_{\lambda},\pm\frac{\lambda}{2}}$, since $u_{\lambda}$ solves \eqref{pb} for $\lambda>0$, we have
\begin{eqnarray*}
\text{Trace}\left(\nabla^2u_{\lambda}(\xi)\right)=\Delta u_{\lambda}(\xi)\not=0.
\end{eqnarray*}
Thus, $\text{Ker}^{\perp}\left(\nabla^2u_{\lambda}(\xi)\right)\not=\emptyset$, where $\text{Ker}\left(\nabla^2u_{\lambda}(\xi)\right)$ is the kernel of the matrix $\nabla^2u_{\lambda}(\xi)$ and the orthogonality is in $\mathbb{R}^6$.  We use $\{e_1,e_2,\cdots e_6\}$ to denote the orthogonal system on $\mathbb{S}^{5}$ such that $\{e_1,e_2,\cdots e_l\}$ and $\{e_{l+1}, e_{l+2}\cdots e_6\}$ are the basis of $\text{Ker}^{\perp}\left(\nabla^2u_{\lambda}(\xi)\right)$ and $\text{Ker}\left(\nabla^2u_{\lambda}(\xi)\right)$, respectively, where $1\leq l\leq 6$.  We also introduce the sets
\begin{eqnarray*}
\Omega_{v_{\lambda},\pm\frac{1}{2}}=\left\{x\in\Omega_{u_{\lambda},\pm\frac{\lambda}{2}}\mid v_{\lambda}(x)=\pm\frac{1}{2}\right\}
\end{eqnarray*}
and
\begin{eqnarray*}
\Omega_{u_{\lambda},v_{\lambda},\pm}=\left\{\xi\in\Omega_{u_{\lambda},\pm\frac{\lambda}{2}}\mid \nabla v_{\lambda}(\xi)\in\text{Ker}^{\perp}\left(\nabla^2u_{\lambda}(\xi)\right)\right\}.
\end{eqnarray*}
We say that $u_{\lambda}$ is an {\bf essentially non-degenerate} solution of the $6D$ Brezis-Nirenberg problem~\eqref{pb} in the non-degenerate domain $\Omega$ if
\begin{eqnarray}\label{eqnPreWu0009}
\Omega_{u_{\lambda},*}:=\Omega_{u_{\lambda},+}\cup \Omega_{u_{\lambda},-}\not=\emptyset,
\end{eqnarray}
where
\begin{eqnarray}\label{eqnPreWu0010}
\Omega_{u_{\lambda},\pm}=\Omega_{u_{\lambda},\pm\frac{\lambda}{2}}\setminus\left(\Omega_{v_{\lambda},\pm\frac{1}{2}}\cap\Omega_{u_{\lambda},v_{\lambda},\pm}\right).
\end{eqnarray}
Our first result introduces the sufficient and necessary condition of a solution $u_\lambda$ of \eqref{pb} to be essentially non-degenerate.
\begin{theorem}\label{thm1}
Let $u_{\lambda}$ be a solution of the $6D$ Brezis-Nirenberg problem~\eqref{pb} in the non-degenerate domain $\Omega$.  Then $u_{\lambda}$ is essentially non-degenerate if and only if $\Omega_{u_{\lambda},\frac{\lambda}{2}}\cup\Omega_{u_{\lambda},-\frac{\lambda}{2}}\not=\emptyset$.
\end{theorem}

\vskip0.12in

In the ODE setting, that is, $\Omega=\mathbb{B}(0,1)$ is the unit ball in $\mathbb{R}^6$, it is well known that
\begin{eqnarray*}
\mathbb{B}(0,1)_{u_{\lambda_0},v_{\lambda_0},+}=\mathbb{B}(0,1)_{u_{\lambda_0},\frac{\lambda_0}{2}}=\{0\}\quad\text{and}\quad\mathbb{B}(0,1)_{u_{\lambda_0},v_{\lambda_0},-}=\mathbb{B}(0,1)_{u_{\lambda_0},-\frac{\lambda_0}{2}}=\emptyset,
\end{eqnarray*}
where $u_{\lambda_0}$ is the unique positive ground-state solution of the $6D$ Brezis-Nirenberg problem~\eqref{pb} in $\Omega=\mathbb{B}(0,1)$.  Moreover, it has been proved in \cite[Proposition~4.2]{AGGPV}, by adapting the ODE argument, that
\begin{eqnarray*}
\mathbb{B}(0,1)_{v_{\lambda_0},\frac{1}{2}}=\emptyset,
\end{eqnarray*}
which implies that $u_{\lambda_0}$, the unique positive ground-state solution of the $6D$ Brezis-Nirenberg problem~\eqref{pb} in $\Omega=\mathbb{B}(0,1)$, is essentially non-degenerate.  Thus, Theorem~\ref{thm1} can be seen as the generalization of \cite[Proposition~4.2]{AGGPV} from the positive ground-state solution in the ODE setting to arbitrary solutions in the PDE setting.  Clearly, to prove Theorem~\ref{thm1} for arbitrary solutions in the PDE setting, we need a PDE argument.  The {\it crucial} trick in this proof is the introduction of the parameter $\eta\in\bbr^N$, which allows us to construct a smooth solution $w_{\eta_0}(\xi)$ of the equation~\eqref{eqnPreWu0027} such that $ w_{\eta_0}(\xi)=0$ and $\nabla w_{\eta_0}(\xi)=0$.

\vskip0.12in

It is known in \cite[Proposition~3.2]{PV2021} that for every bounded domain $\Omega$, there exists $\lambda_0\in (0, \lambda_1)$ such that $\Omega_{u_{\lambda_0},\frac{\lambda_0}{2}}\not=\emptyset$, where $u_{\lambda_0}$ is the positive ground-state solution of the $6D$ Brezis-Nirenberg problem~\eqref{pb} and $\lambda_1:=\lambda_1(\Omega)$ is the first eigenvalue of the Laplacian operator $-\Delta$ with Dirichlet boundary condition.  Thus, by Theorem~\ref{thm1}, the $6D$ Brezis-Nirenberg problem~\eqref{pb} in the non-degenerate domain $\Omega$ always has a solution which is essentially non-degenerate.

\vskip0.12in

Even though we have not proved that
\begin{eqnarray}\label{nq}
\Omega_{u_{\lambda_0},\frac{\lambda_0}{2}}\setminus\Omega_{v_{\lambda_0},\frac{1}{2}}\not=\emptyset
\end{eqnarray}
for positive ground-state solutions of the Brezis-Nirenberg problem~\eqref{pb} on any non-degenerate domains, as excepted in \cite{PV2021}, Theorem~\ref{thm1} is still good enough to improve the main results in \cite{PV2021}.

\vskip0.12in

In what follows, we let $\Omega$ be a non-degenerate domain of the $6D$ Brezis-Nirenberg problem~\eqref{pb} and $\lambda_0>0$ such that the $6D$ Brezis-Nirenberg problem~\eqref{pb} in the non-degenerate domain $\Omega$ has a solution $u_{\lambda_0}$ with $\Omega_{u_{\lambda_0},\frac{\lambda_0}{2}}\cup\Omega_{u_{\lambda_0},-\frac{\lambda_0}{2}}\not=\emptyset$.  Moreover, by classical regularity theorem of the elliptic equations, we also know that there are only finitely many points in $\Omega_{u_{\lambda_0},\frac{\lambda_0}{2}}\cup\Omega_{u_{\lambda_0},-\frac{\lambda_0}{2}}$.  We denote the number of points in $\Omega_{u_{\lambda_0},\frac{\lambda_0}{2}}\cup\Omega_{u_{\lambda_0},-\frac{\lambda_0}{2}}$ by $k(u_{\lambda_0})$.

\vskip0.12in

We also denote by $v_{\lambda_0}$ the unique solution of the problem \eqref{eqnPreWu0026} with $\lambda=\lambda_0$ and by $w_{\lambda_0}$ the unique solution of the following problem
\begin{eqnarray}\label{eqnPreWu0029}
\left\{
\aligned
&-\Delta w=\left(\lambda+2|u_{\lambda_0}|\right)w+v_{\lambda_0}+\text{sgn}(u_{\lambda_0}(x))v_{\lambda_0}^2,\quad\text{in }\Omega,\\
&v=0\quad\text{on }\partial\Omega,
\endaligned
\right.
\end{eqnarray}
Besides, as in \cite{LVWW}, we denote the standard Aubin-Talenti bubble for every $\mu>0$ and $\xi\in\Omega$ by $\mathcal{U}_{\mu,\xi}$ while, we denote its projection into $H^1_0(\Omega)$ by $\mathcal{W}_{\mu,\xi}$.  It is well known that (see \cite{A}, \cite{CGS}, \cite{T}) the Aubin-Talenti bubble $\mathcal{U}_{\mu,\xi}$ is explicitly given by
\begin{eqnarray}\label{talenti}
\mathcal{U}_{\mu,\xi}=\alpha_N\mu^{\frac{N-2}{2}}\left(\frac{1}{\mu^{2}+|x-\xi|^2}\right)^{\frac{N-2}{2}}=\mu^{-\frac{N-2}{2}}\mathcal{U}_{1,0}\left(\frac{x-\xi}{\mu}\right)
\end{eqnarray}
with $\alpha_N=[N(N-2)]^{\frac{N-2}{4}}$.  Now, our main result in this paper reads as
\begin{theorem}\label{thm2}
Let $\Omega$ be a non-degenerate domain of the $6D$ Brezis-Nirenberg problem~\eqref{pb} and $\lambda_0>0$ such that the $6D$ Brezis-Nirenberg problem~\eqref{pb} in the non-degenerate domain $\Omega$ has an essentially non-degenerate solution $u_{\lambda_0}$.  Then for any $\{\xi_{j,0}\}_{1\leq j\leq k}\subset\Omega_{u_{\lambda_0},*}$ with $1\leq k\leq k(u_{\lambda_0})$, where $\Omega_{u_{\lambda_0},*}$ is given by \eqref{eqnPreWu0009}, the $6D$ Brezis-Nirenberg problem~\eqref{pb} in the non-degenerate domain $\Omega$ has a solution of the form
\begin{eqnarray}\label{eqnPreWu0001}
u_{\lambda}=u_{\lambda_0}+\sum_{j=1}^{k}\beta_j\mathcal{W}_{\mu_j(\lambda),\xi_j(\lambda)}+\ve v_{\lambda_0}+\ve^2w_{\lambda_0}+\varphi_{\lambda},
\end{eqnarray}
where $v_{\lambda_0}$ and $w_{\lambda_0}$ are given by \eqref{eqnPreWu0026} and \eqref{eqnPreWu0029}, respectively, $\ve=\lambda-\lambda_0$, $\|\nabla\varphi_{\lambda}\|_2\to0$ as $\ve\to0$ and $\lim_{\ve\to0}\max_{1\leq j\leq k}\{\mu_j(\lambda)\}=0$ with $\liminf_{\lambda\to\lambda_0}\frac{\max_{1\leq j\leq k}\{\mu_j(\lambda)\}}{\min_{1\leq j\leq k}\{\mu_j(\lambda)\}}>0$.
Moreover,
\begin{eqnarray*}
\beta_j=\left\{\aligned
&-1,\quad &\xi_{j,0}\in\Omega_{u_{\lambda_0},+},\\
&1,\quad &\xi_{j,0}\in\Omega_{u_{\lambda_0},-},
\endaligned
\right.
\end{eqnarray*}
where $\Omega_{u_{\lambda_0},\pm}$ are given by \eqref{eqnPreWu0010},
and
\begin{enumerate}
\item[$(1)$]\quad $\xi_j(\lambda)=\xi_{j,0}$, $\frac{\ve}{|\ve|}=-\text{sgn}\left(\frac{1}{2}+\beta_jv_{\lambda_0}(\xi_{j,0})\right)$ and $\mu_j(\lambda)=|\ve|\tau_j$ with $\tau_j\to\tau_{j,0}$ as $\ve\to0$ and $\tau_{j,0}$ being the unique global minimum point of the function
\begin{eqnarray*}
P_j(\tau)=-d_1\left|\frac{1}{2}+\beta_jv_{\lambda_0}(\xi_j)\right|\tau^2+\frac{11}{9}d_2\tau^3
\end{eqnarray*}
for $\xi_{j,0}\not\in\left(\Omega_{v_{\lambda_0},\frac{1}{2}}\cup\Omega_{v_{\lambda_0},-\frac{1}{2}}\right)$,
\item[$(2)$]\quad $\xi_j(\lambda)=\xi_{j,0}+|\ve|^{s}e_{j,l_{j}+1}$ with $e_{j,l_{j}+1}\in\mathbb{S}^5\cap\text{Ker}\left(\nabla^2u_{\lambda_0}(\xi_{j,0})\right)$ and $\frac{1}{2}<s<1$, $\frac{\ve}{|\ve|}=-\text{sgn}\left(\nabla v_{\lambda_0}(\xi_{j,0})\cdot e_{j,l_{j}+1}\right)$, $\mu_j(\lambda)=|\ve|^{1+s}\tau_j$ with $\tau_j\to\tau_{j,0}$ as $\ve\to0$ and $\tau_{j,0}$ being the unique global minimum point of the function
\begin{eqnarray*}
Q_j(\tau)=-d_1\left|\nabla v_{\lambda_0}(\xi_{j,0})\cdot e_{j,l_{j}+1}\right|\tau^2+\frac{11}{9}d_2\tau^3
\end{eqnarray*}
for $\xi_{j,0}\in\left(\Omega_{v_{\lambda_0},\frac{1}{2}}\setminus\Omega_{u_{\lambda_0},v_{\lambda_0},+}\right)\cup\left(\Omega_{v_{\lambda_0},-\frac{1}{2}}\setminus\Omega_{u_{\lambda_0},v_{\lambda_0},-}\right)$,
\end{enumerate}
where $d_1,d_2>0$ are constants with precise formulas.
\end{theorem}

Theorem \ref{thm2} is the generalization of \cite[Theorem 1.1]{PV2021} where such bubbling solutions were only constructed for the one-bubble case with the condition $2v_{\lambda_0}(\xi_0)\neq 1$ on the blow-up point.  For the multi-bubble case, if the concentration point $\xi_{j,0}$ still satisfies the condition $2v_{\lambda_0}(\xi_{j,0})\neq 1$, then we can exactly follow the arguments of \cite[Theorem 1.1]{PV2021} to complete the construction, by using the same ansatz as in \cite[Theorem 1.1]{PV2021}, with a slightly complicated computation which comes from the interaction between different bubbles.  However, since we are not able to prove that there always exists $\xi_0\in\Omega_{u_{\lambda_0},\pm\frac{\lambda_0}{2}}$ such that $2v_{\lambda_0}(\xi_{0})\neq \pm1$, we have to improve the ansatz in \cite[Theorem 1.1]{PV2021} to the next order term in \eqref{eqnPreWu0001}.  Our {\it key} observation is that by using the new ansatz \eqref{eqnPreWu0001} to expand the energy functional to the next order term, the leading order term in the energy expansion of the reduced problem will not vanish for essentially non-degenerate solutions, which always exist for generic bounded domains.

By Theorem~\ref{thm2}, we can see that the parameters of the constructed bubbling solutions of the $6D$ Brezis-Nirenberg problem~\eqref{pb} are implicitly coupled in the sense that the blow-up rates depend on the concentration points.  Since we have observed in \cite{LVWW} that the parameters of the constructed bubbling solutions of the $5D$ Brezis-Nirenberg problem~\eqref{pb01} are also implicitly coupled while, the parameters of the constructed bubbling solutions of the $4D$ Brezis-Nirenberg problem~\eqref{pb01} are fully coupled.  Thus, it seems that the dimension $N=5$ is {\it critical} in constructing bubbling solutions of the Brezis-Nirenberg problem which concentrates at a positive value $\lambda$ for $N\geq4$.

For any essentially non-degenerate solution $u_{\lambda_0}$ of the $6D$ Brezis-Nirenberg problem~\eqref{pb}, $u_{\lambda_0}$ is degenerate at any point in $\left(\Omega_{u_{\lambda_0},\frac{\lambda_0}{2}}\setminus\Omega_{u_{\lambda_0},v_{\lambda_0},+}\right)\cup\left(\Omega_{u_{\lambda_0},-\frac{\lambda_0}{2}}\setminus\Omega_{u_{\lambda_0},v_{\lambda_0},-}\right)$ by the definitions.  This is the reason that a free parameter $s$ is involved in the $(2)$ case of Theorem~\ref{thm2}.

\subsection{Open questions of \eqref{pb01}}
Even though the Brezis-Nirenberg problem~\eqref{pb01} has a long history since 1983 and many different results are available for it, there are still some gaps continue to make this problem challenging, which leaves many open questions, see, for example the interesting questions proposed in \cite{PV2021,SWY}.  Here, we will also list some other {\bf open questions}  about the Brezis-Nirenberg problem~\eqref{pb01}.

\subsubsection{Variational proofs of Brezis' first open problem}
Through a highly intricate and technical construction, Sun, Wei, and Yang provided a solution to Brezis’ first open problem. It is natural to ask
\begin{enumerate}
\item[] {\bf Open question~1}: Is there a more direct and possible variational proof of Brezis' first open problem?
\end{enumerate}

\subsubsection{Infinitely many sign-changing solutions of the Brezis-Nirenberg problem~\eqref{pb01}.}
In \cite{SWY}, it was proved that \eqref{pb01} in the unit ball $\Omega=\mathbb{B}(0,1)$ has infinitely many non-radial sign-changing solutions with arbitrary large energy for every $\lambda>0$.  The strategy in proving this result is to use the Kelvin transformation of the ``crown solution'', constructed by del Pino-Musso-Pacard-Pistoia in \cite{DMPP2011}, as the building block to disingularize two circles which are very colse to each other and the unit sphere.  When the domain is a three-dimensional cuboid (rectangular prism), Sun, Wei, and Yang \cite{SWY2} proved the existence of infinitely many sign-changing solutions. Their approach was to populate the domain with a large number of crown-type solutions.
Since there is no reason to think that this strategy will be invalid in general bounded domains, we want to ask
\begin{enumerate}
\item[] {\bf Open question~2}:\quad In a general bounded domain $\Omega$, if necessary, it could be convex, are there infinitely many   sign-changing solutions of \eqref{pb01} with arbitrary large energy for any $\lambda>0$?
\end{enumerate}

In dimension 4, the existence of  sign-changing solutions was proved for every $\lambda >0, \lambda \not \in \sigma (-\Delta)$ (\cite{CFP}). The missing case is
\begin{enumerate}
\item[] {\bf Open question~3}:\quad when $N=4, \lambda \in \sigma (-\Delta)$, are there sign-changing solutions for every $\lambda>0$?
\end{enumerate}

In \cite{SZ}, Schechter and Zou proved that for $N\geq 7, \lambda >0$ there are infinitely many sign-changing solutions. A natural equation is
\begin{enumerate}
\item[] {\bf Open equestion~4}: \quad In a general bounded domain, are there infinitely many sign-changing solutions  when $N\leq 6$?
\end{enumerate}

\subsubsection{Blow-up solutions of the Brezis-Nirenberg problem~\eqref{pb01} with finite energy.}
Since there is a $\lambda>0$ in the Brezis-Nirenberg problem~\eqref{pb01}, it is natural to consider $\lambda$ to be a parameter.  We say that $u_\lambda$ is a blow-up solution of the Brezis-Nirenberg problem~\eqref{pb01} with finite energy at a fixed parameter $\overline{\lambda}$ if $\|u_\lambda\|_{L^\infty(\Omega)}\to+\infty$ and $E_\lambda(u_\lambda)\lesssim1$ as $\lambda\to\overline{\lambda}$, where $E_\lambda(u)$ is the energy functional of the Brezis-Nirenberg problem~\eqref{pb01}.
Correspondingly, $\overline{\lambda}$ is called a concentration value with finite energy of the Brezis-Nirenberg problem~\eqref{pb01}.  In \cite{LVWW} and this paper, blow-up solutions of the Brezis-Nirenberg problem~\eqref{pb01} with finite energy at a fixed parameter $\overline{\lambda}>0$ for $4\leq N\leq 6$ are constructed.  The ansatz of the constructions in \cite{LVWW} and this paper is very different.  Taking into account the asymptotic behavior of blow-up solution of the Brezis-Nirenberg problem~\eqref{pb01} with finite energy at a fixed parameter in the radial setting (cf. \cite{AGGPV,IP1}), it is natural to think that the ansatz of the construction of the blow-up solution of the Brezis-Nirenberg problem~\eqref{pb01} with finite energy at a fixed parameter in dimension $N=3$ will be very different from that of $4\leq N\leq6$, and the blow-up solution of the Brezis-Nirenberg problem~\eqref{pb01} with finite energy can only exist at $\overline{\lambda}=0$ for $N\geq7$.  Thus, we also want to ask
\begin{enumerate}
\item[] {\bf Open question~5}:\quad In a general bounded domain $\Omega$, if necessary, it could be convex, is it possible to construct blow-up solutions of the Brezis-Nirenberg problem~\eqref{pb01} with finite energy at a fixed parameter which is different from $\lambda^*$ predicted by Druet in \cite{D} in dimension $N=3$?
\item[] {\bf Open question~6}:\quad In a general bounded domain $\Omega$, if necessary, it could be convex, does there exist blow-up solutions of the Brezis-Nirenberg problem~\eqref{pb01} with finite energy at a fixed parameter $\overline{\lambda}>0$ in dimension $N\geq7$?
\end{enumerate}
Besides, as far as we know the complete scenario of sign-changing solutions has been obtained only in the ODE setting by Esposito, Ghoussoub, Pistoia and Vaira
in \cite{EGPV} in dimensions $N\geq 7$ and by Amadori, Gladiali, Grossi, Pistoia and Vaira \cite{AGGPV} in lower dimensions $3\leq N\leq 6$, which completes the previous works of Atkinson, Brezis and Peletier \cite{ABP} and Iacopetti and Pacella \cite{IP1,IP}.  Thus, we also want to ask
\begin{enumerate}
\item[] {\bf Open question~7}:\quad In a general bounded domain $\Omega$, if necessary, it could be convex or even strictly star-sharped, what is the asymptotic behavior of blow-up solutions of the Brezis-Nirenberg problem~\eqref{pb01} with finite energy at a fixed parameter $\overline{\lambda}>0$?
\end{enumerate}

\subsubsection{Boundary concentrations of the Brezis-Nirenberg problem~\eqref{pb01}.}
In \cite{MRV,V}, the boundary concentration of sign-changing solutions of the Brezis-Nirenberg problem~\eqref{pb01} has been constructed as $\lambda\to0$ in dimension $N\geq7$.  This boundary concentration is observed by looking at the solution of the form:
\begin{eqnarray}\label{bd01}
u_{\lambda} \approx u_{0}-\sum_{j=1}^{k}\mathcal{W}_{\mu_j(\lambda),\xi_j(\lambda)}
\end{eqnarray}
where $u_{0}$ is a non-degenerate nontrivial solution of the Brezis-Nirenberg problem~\eqref{pb01} for $\lambda=0$.  We recall that the Brezis-Nirenberg problem~\eqref{pb01} for $\lambda=0$ only has trivial solution if $\Omega$ is strictly star-sharped, thus, the boundary concentration can only observed by at least assuming that $\Omega$ is not strictly star-sharped.
Since the ansatz in \eqref{bd01} is almost the same as that of \eqref{eqnPreWu0001}, we want to ask
\begin{enumerate}
\item[] {\bf Open question~8}:\quad In a general bounded domain $\Omega$, if necessary, it could have nontrivial topology, is it possible to construct boundary concentration of sign-changing solutions of the Brezis-Nirenberg problem~\eqref{pb01} at a fixed parameter $\lambda\geq0$ in dimensions $N\geq6$?
\end{enumerate}
On the other hand, the blow-up analysis in \cite{D2003} reveals that a competitive mechanism between contributions of nonzero weak limit and linear term can only exist in dimension $N=6$ and the ansatz of the constructions for $N=3$, $4\leq N\leq5$, $N=6$ and $N\geq7$ are very different, thus, we also want to ask
\begin{enumerate}
\item[] {\bf Open question~9}:\quad In a general bounded domain $\Omega$, if necessary, it could have nontrivial topology, does there exist boundary concentration of sign-changing solutions of the Brezis-Nirenberg problem~\eqref{pb01} at a fixed parameter $\lambda\geq0$ in dimensions $3\leq N\leq5$?
\end{enumerate}

\subsubsection{Bubble-tower solutions of the Brezis-Nirenberg problem~\eqref{pb01}.}
Bubble-tower solutions of the Brezis-Nirenberg problem~\eqref{pb01} for $N\geq7$ as $\lambda\to0$ were constructed in \cite{IV1,Pre}.  Moreover, it has been proved by Iacopetti and Pacella in \cite{IP} that the Brezis-Nirenberg problem~\eqref{pb01} has no bubble-tower solutions, with two bubbles blowing up at the same point, in low dimensions $4\leq N\leq 6$ as $\lambda\to0$.  Thus, we propose the following conjecture.
\begin{enumerate}
\item[] {\bf Conjecture~1}:\quad In a general bounded domain $\Omega$, bubble-tower solutions of the Brezis-Nirenberg problem~\eqref{pb01} can only exist for $N\geq7$ as $\lambda\to0$.
\end{enumerate}
Besides, there is no analysis of bubble-tower solutions of the Brezis-Nirenberg problem~\eqref{pb01} as $\lambda\to\overline{\lambda}>0$ for $N\geq4$ ($\lambda>\lambda_*$ for $N=3$), thus, we also want to ask
\begin{enumerate}
\item[] {\bf Open question~10}:\quad In a general bounded domain $\Omega$, does there exist bubble-tower solutions of sign-changing solutions of the Brezis-Nirenberg problem~\eqref{pb01} at a fixed paramter $\lambda>0$ for $N\geq4$ ($\lambda>\lambda_*$ for $N=3$)?
\end{enumerate}

\subsubsection{Least-energy sign-changing solutions of the Brezis-Nirenberg problem~\eqref{pb01}.}
Let
\begin{eqnarray}\label{eq0003}
E_\lambda(u)=\frac{1}{2}\|\nabla u\|^2_{L^2(\Omega)}-\frac{\lambda}{2}\|u\|^2_{L^2(\Omega)}-\frac{1}{2^*}\|u\|^{2^*}_{L^{2^*}(\Omega)}.
\end{eqnarray}
be the corresponding energy functional of the Brezis-Nirenberg problem~\eqref{pb01}  and
\begin{eqnarray*}
\mathcal{N}_{sg}=\left\{u\in H_0^1(\Omega)\backslash\{0\}\mid E'_\lambda(u^{\pm})u^{\pm}=0\text{ and }u^\pm\not=0\right\}
\end{eqnarray*}
be its associated sign-changing Nehari manifold. It is proved in \cite{SWW} that there exists $\hat{u}\in \mathcal{N}_{sg}$ such that
\begin{eqnarray*}
E_\lambda(\hat{u})=\inf_{\mathcal{N}_{sg}}E_\lambda(u)\quad\text{and}\quad E'_\lambda(\hat{u})=0\text{ in }H^{-1},
\end{eqnarray*}
where $H^{-1}$ is the dual space of $H^1_0(\Omega)$, provided
\begin{enumerate}
\item[(1)]\quad $N=4$ and $\lambda>0$ with $\lambda\not\in\sigma(-\Delta)$,
\item[(2)]\quad $N\geq5$ and $\lambda>0$.
\end{enumerate}
$\hat{u}$ is called the least-energy sign-changing solutions of the Brezis-Nirenberg problem~\eqref{pb01}.
Taking into account the Pohozaev identity, we want to ask
\begin{enumerate}
\item[] {\bf Open question~11}:\quad In a general bounded domain $\Omega$, if necessary, it could be convex or even strictly star-sharped, does there exist least-energy sign-changing solutions of the Brezis-Nirenberg problem~\eqref{pb01} for $N=4$ and $\lambda\in\sigma(-\Delta)$?
\end{enumerate}
Since it has been pointed out in \cite{AP} that the existence of least-energy sign-changing solutions of the Brezis-Nirenberg problem~\eqref{pb01} for $N=3$ is still unknown, we also want to ask
\begin{enumerate}
\item[] {\bf Open question~12}:\quad In a general bounded domain $\Omega$, if necessary, it could be convex or even strictly star-sharped, does there exist least-energy sign-changing solutions of the Brezis-Nirenberg problem~\eqref{pb01} for $N=3$ and $\lambda>0$?
\end{enumerate}

\subsubsection{Related studies of the Brezis-Nirenberg type problems.}
There are lots of Brezis-Nirenberg type problems, such as the fractional version:
\begin{equation}\label{frac-pb01}
\left\{\begin{aligned}
&(-\Delta)^s u = \lambda u+|u|^{2_s^*-2}u, \quad &\mbox{in}\,\Omega,\\
&u=0,\quad &\mbox{on}\, \partial\Omega,
\end{aligned}\right.
\end{equation}
where $0<s<1$ and $2_s^*=\frac{2N}{N-2s}$, the high order version:
\begin{equation}\label{high-pb01}
\left\{\begin{aligned}
&(-\Delta)^m u = \lambda u+|u|^{2_m^*-2}u, \quad &\mbox{in}\,\Omega,\\
&u=0,\quad &\mbox{on}\, \partial\Omega,
\end{aligned}\right.
\end{equation}
where $m\in\mathbb{N}$ with $m\geq2$ and $2_m^*=\frac{2N}{N-2m}$, the $p$-Laplacian version:
\begin{equation}\label{p-pb01}
\left\{\begin{aligned}
&-\Delta_p u = \lambda u+|u|^{2_p^*-2}u, \quad &\mbox{in}\,\Omega,\\
&u=0,\quad &\mbox{on}\, \partial\Omega,
\end{aligned}\right.
\end{equation}
where $\Delta_p u=\text{div}(|\nabla u|^{p-2}\nabla u)$ with $1<p<N$ and $2_p^*=\frac{pN}{N-p}$, and so on.  The studies on these Brezis-Nirenberg type problems also have a rich history which we do not want to review in this paper.  Thus, we also want to ask
\begin{enumerate}
\item[] {\bf Open question~13}:\quad Could the very recent results of the Brezis-Nirenberg problem~\eqref{pb01} be generalized to these Brezis-Nirenberg type problems?
\end{enumerate}

\subsection{Notations} Let
\begin{eqnarray*}
\langle u,v \rangle:=\into \nabla u\nabla v dx,\qquad ||u||:=\left(\into|\nabla u|^2 dx\right)^{\frac{1}{2}}
\end{eqnarray*}
be the inner product in $H_0^1(\Omega)$ and its corresponding norm.  Moreover, we define
\begin{eqnarray*}
||u||_{L^r}:=\left(\into |u|^rdx\right)^{\frac{1}{r}},\quad \langle u,v \rangle_{L^2}:=\into u v dx
\end{eqnarray*}
as the $L^r(\Omega)$-standard norm for any $r\in[1,+\infty]$.  Moreover, when $A \neq \Omega$ is any Lebesgue measurable subset of $\bbr^N$, we use the alternative notations $\|u\|_A$, $\|u\|_{L^r,A}$.

\section{Setting of the problem and the ansatz}
\subsection{Preliminaries} Let $(-\Delta)^{-1}:L^{\frac 32}(\Omega) \to H^1_0(\Omega)$ be the operator defined as the unique solution of the equation $-\Delta u =v$ in $\Omega$ and $u=0$ on $\partial\Omega$.  By the H\"older inequality it follows that $$\|(-\Delta)^{-1}(v)\|\leq C \|v\|_{L^{\frac 32}}, \quad \mbox{for all}\,\, v\in L^{\frac 32}(\Omega)$$ for some positive $C$ independent of $v$.  Hence we rewrite the problem \eqref{pb} as
\begin{equation}\label{pb1}
u=(-\Delta)^{-1}\left[f(u)+(\lambda_0+\ve) u\right],\quad u\in H^1_0(\Omega)\end{equation}
with $f(u)=|u|u$.  We denote by $G(x, y)$ the Green's function of the Laplace operator given by $$G(x, y)=\frac{1}{4\omega_6}\left(\frac{1}{|x-y|^4}-H(x, y)\right),$$ where $\omega_6$ denotes the surface area of the unit sphere in $\bbr^6$ and $H$ is the regular part of $G$, namely $$\Delta H(x, y)=0,\quad \mbox{in}\,\, \Omega,\qquad H(x, y)=\frac{1}{|x-y|^4},\quad\mbox{on}\,\, \partial\Omega.$$ In the previous section we have denoted by $\mathcal{W}_{\mu,\xi}$ the projection of the bubble $\mathcal{U}_{\mu,\xi}$ onto $H^1_0(\Omega)$. It is well known that the following expansion holds (see \cite{R})
\be\label{expW}
\mathcal{W}_{\mu,\xi}(x)=\mathcal{U}_{\mu,\xi}-\alpha_6\mu^2 H(x, \xi)+\mathcal O(\mu^4),\quad \mbox{as}\,\mu\to 0\ee
uniformly with respect to $\xi$ in compact sets of $\Omega$. Moreover (see \cite{BE}) it is also known that every solution to the linear equation
\begin{equation*}\label{linearpblim}
-\Delta \Psi_{\mu, \xi}=2\mathcal{U}_{\mu,\xi}\Psi_{\mu, \xi},\quad \mbox{in}\,\, \mathbb R^6\end{equation*}
is a linear combination of the functions $\Psi^i_{\mu, \xi}$ $(i=0, \ldots, 6)$ given by
\begin{equation*}\label{psi0}
\Psi^0_{\mu, \xi}(x)=\partial_{\mu} \mathcal{U}_{\mu,\xi}=2\alpha_6 \mu \frac{|x-\xi|^2-\mu^2}{(\mu^2+|x-\xi|^2)^3}\end{equation*}
and
\begin{equation*}\label{psii}
\Psi^i_{\mu, \xi}(x)=\partial_{\xi^i} \mathcal{U}_{\mu,\xi}=4\alpha_6 \mu^2 \frac{x^i-\xi^i}{(\mu^2+|x-\xi|^2)^3}, \quad i=1, \ldots, 6.\end{equation*}
We denote by $\mathcal{Z}_{\mu,\xi}^i$ the projection of $\Psi^i_{\mu, \xi}$ onto $H^1_0(\Omega)$, i.e.
\begin{equation*}\label{Zi}
-\Delta \mathcal{Z}_{\mu,\xi}^i=2 \mathcal{U}_{\mu,\xi}\Psi^i_{\mu, \xi},\quad \mbox{in}\,\, \Omega,\qquad \mathcal{Z}_{\mu,\xi}^i=0,\quad \mbox{on}\,\,\partial\Omega.\end{equation*}
Elliptic estimates give
\begin{equation*}\label{expZ0}
\mathcal{Z}_{\mu,\xi}^0(x)=\Psi^0_{\mu, \xi}-2\mu\alpha_6 H(x, \xi)+\mathcal O(\mu^3),\quad \mbox{as}\,\, \mu\to 0
\end{equation*}
and
\begin{equation*}\label{expZi}
\mathcal{Z}_{\mu,\xi}^i(x)=\Psi^i_{\mu, \xi}-\mu^2\alpha_6 \partial_{\xi^i}H(x, \xi)+\mathcal O(\mu^4),\quad \mbox{as}\,\, \mu\to 0,\,\, i=1, \ldots, 6
\end{equation*}
uniformly with respect to $\xi$ in compact sets of $\Omega$.

\subsection{The ansatz}
In what follows we look for a solution of \eqref{pb} of the form
\be\label{ansatz}
u_\lambda(x):=u_{\lambda_0}(x)+\ve v_{\lambda_0}(x)+\ve^2 w_{\lambda_0}(x) +\sum_{j=1}^k \beta_j \mathcal W_{\mu_j, \xi_j}(x)+\varphi_\lambda(x)\ee
where
\be\label{muj}
\mu_j=\tau_j \bar\mu\quad \mbox{with}\,\, \lim_{\lambda\to \lambda_0}\tau_j:=\tau_{j, 0}>0\,\, \mbox{being constants}\ee
and $\bar\mu>0$ with $\lim\limits_{\lambda\to\lambda_0}\bar\mu= 0$.  Moreover
\be\label{xij}
\xi_j=\left\{\begin{array}{ll} \xi_{j, 0},\quad & \mbox{if}\,\, \xi_{j, 0}\in\Omega_{u_{\lambda_0}, \pm\frac{\lambda_0}{2}}\setminus \Omega_{v_{\lambda_0}, \pm\frac1 2},\\
\xi_{j, 0}+\rho e_{j, l_j+1}, \quad & \mbox{if}\,\, \xi_{j, 0}\in\left(\Omega_{u_{\lambda_0}, \pm\frac{\lambda_0}{2}}\setminus \Omega_{u_{\lambda_0}, v_{\lambda_0}, \pm}\right)\cap \Omega_{v_{\lambda_0}, \pm \frac 12}.\\
\end{array}\right.\ee
We remark that $\mu_j=\mu_j(\lambda)$ and $\xi_j=\xi_j(\lambda)$.  Without loss of generality, we assume also that
\begin{equation*}\label{muj1}0<\mu_1\leq\mu_2\leq\cdots\leq\mu_k.\end{equation*}
Finally, $\varphi_\lambda$ is a remainder term which is small as $\lambda\to\lambda_0$ which belongs, as usual, to the space $\mathcal K^\bot_{\boldsymbol\mu, \boldsymbol\xi}$ defined as follows.  Let $$\mathcal K_{\boldsymbol\mu, \boldsymbol\xi}:={\rm span}\left\{\mathcal Z^i_{\mu_j, \xi_j}\,\,:\,\, j=1, \ldots , k, \,\, i=0, \ldots, 6\right\}$$ and
$$\mathcal K^\bot_{\boldsymbol\mu, \boldsymbol\xi}:=\left\{\varphi\in H^1_0(\Omega)\,\,:\,\, \langle \varphi, \mathcal Z^i_{\mu_j, \xi_j}\rangle =0\,\, i=0, \ldots, 6\,\, j=1, \ldots k\right\}.$$
Let $\Pi_{\boldsymbol\mu, \boldsymbol\xi}$ and $\Pi^\bot_{\boldsymbol\mu, \boldsymbol\xi}$ be the projections of $H^1_0(\Omega)$ onto $\mathcal K_{\boldsymbol\mu, \boldsymbol\xi}$ and $\mathcal K_{\boldsymbol\mu, \boldsymbol\xi}^\bot$ respectively.  Then, solving \eqref{pb1}, is equivalent to solve the system
\begin{numcases}{}
\Pi^\bot_{\boldsymbol\mu, \boldsymbol\xi}\left\{u_\lambda(x)-(-\Delta)^{-1}\left[f(u_\lambda)+\lambda u_\lambda\right]\right\}=0,\label{ort}\\
\Pi_{\boldsymbol\mu, \boldsymbol\xi}\left\{u_\lambda(x)-(-\Delta)^{-1}\left[f(u_\lambda)+\lambda u_\lambda\right]\right\}=0,\label{bif}
\end{numcases}
where $u_{\lambda}(x)$ is given by \eqref{ansatz}. In what follows, first we will solve \eqref{ort} to find the remainder term $\varphi_{\lambda}(x)$ and then we study the finite dimensional equation \eqref{bif}.
\section{The Nonlinear Projected problem}
Let us define \begin{equation}\label{sim}\mathcal V_{\boldsymbol\mu, \boldsymbol\xi}:=u_{\lambda_0}(x)+\ve v_{\lambda_0}(x)+\ve^2 w_{\lambda_0}(x) +\sum_{j=1}^k \beta_j \mathcal W_{\mu_j, \xi_j}(x)\end{equation} and \be\label{fix} z_{\lambda_0}(x):=u_{\lambda_0}(x)+\ve v_{\lambda_0}(x)+\ve^2 w_{\lambda_0}(x) . \ee
The equation \eqref{ort} can be rewritten as
\be\label{ort1}
\mathcal L_{\boldsymbol\mu, \boldsymbol\xi}(\varphi_\lambda)+\mathcal R_{\boldsymbol\mu, \boldsymbol\xi}+\mathcal N_{\boldsymbol\mu, \boldsymbol\xi}(\varphi_\lambda)=0\ee where
\be\label{lin}
\mathcal L_{\boldsymbol\mu, \boldsymbol\xi}(\varphi_\lambda):=\Pi^\bot_{\boldsymbol\mu, \boldsymbol\xi}\left\{\varphi_\lambda-(-\Delta)^{-1}\left[f'(\mathcal V_{\boldsymbol\mu, \boldsymbol\xi})\varphi_\lambda+\lambda\varphi_\lambda\right]\right\}\ee is the linearized operator at $\mathcal V_{\boldsymbol\mu, \boldsymbol\xi}$,
\be\label{error}
\mathcal R_{\boldsymbol\mu, \boldsymbol\xi}:=\Pi^\bot_{\boldsymbol\mu, \boldsymbol\xi}\left\{\mathcal V_{\boldsymbol\mu, \boldsymbol\xi}-(-\Delta)^{-1}\left[f(\mathcal V_{\boldsymbol\mu, \boldsymbol\xi})+\lambda\mathcal V_{\boldsymbol\mu, \boldsymbol\xi}\right]\right\}\ee
is the error term while
\begin{equation*}\label{quadratic}
\mathcal N_{\boldsymbol\mu, \boldsymbol\xi}(\varphi_\lambda):=\Pi^\bot_{\boldsymbol\mu, \boldsymbol\xi}\left\{-(-\Delta)^{-1}\left[f(\mathcal V_{\boldsymbol\mu, \boldsymbol\xi}+\varphi_\lambda)-f(\mathcal V_{\boldsymbol\mu, \boldsymbol\xi})-f'(\mathcal V_{\boldsymbol\mu, \boldsymbol\xi})\varphi_\lambda\right]\right\}\end{equation*}
is a quadratic form. To solve \eqref{ort1}, we need first to estimate the $H^1_0(\Omega)-$ norm of the error term $\mathcal R_{\boldsymbol\mu, \boldsymbol\xi}$ given by \eqref{error}.
\begin{lemma}\label{stimaerrore}
For any $(\xi_j)_{j=1, \ldots, k}$ satisfying \eqref{xij} and any $(\mu_j)_{j=1,\ldots, k}$ satisfying \eqref{muj} and for any $|\ve|$ sufficiently small it holds
\begin{equation*}\label{stimaR}
\|\mathcal R_{\boldsymbol\mu, \boldsymbol\xi}\|\lesssim \bar\mu^2|\ln \bar\mu|^{\frac 23}+|\ve|^3.\end{equation*}\end{lemma}
\begin{proof}
By the definition of $\mathcal{V}_{\pmb{\mu,\xi}}$, we have that
$$\begin{aligned}&-\Delta \mathcal V_{\boldsymbol\mu, \boldsymbol\xi}-(\lambda_0+\ve)\mathcal V_{\boldsymbol\mu, \boldsymbol\xi}-f(\mathcal V_{\boldsymbol\mu, \boldsymbol\xi})=-f(\mathcal V_{\boldsymbol\mu, \boldsymbol\xi})+f(u_{\lambda_0})+\sum_{j=1}^k \beta_j \mathcal U_{\mu_j, \xi_j}^2\\
&\qquad-(\lambda_0+\ve) \sum_{j=1}^k \beta_j \mathcal W_{\mu_j, \xi_j}+\ve^2 f'(u_{\lambda_0}) w_{\lambda_0}+\ve^2{\rm sgn}(u_{\lambda_0}(x))v_{\lambda_0}^2-\ve^3 w_{\lambda_0}.\end{aligned}$$
By continuity of $\Pi^\bot_{\boldsymbol\mu, \boldsymbol\xi}$, \eqref{sim} and \eqref{error}, we get that
$$\begin{aligned}\|\mathcal R_{\boldsymbol\mu, \boldsymbol\xi}\|\lesssim& \|-\Delta \mathcal V_{\boldsymbol\mu, \boldsymbol\xi}-(\lambda_0+\ve)\mathcal V_{\boldsymbol\mu, \boldsymbol\xi}-f(\mathcal V_{\boldsymbol\mu, \boldsymbol\xi})\|_{L^{\frac 32}}\\
\lesssim& \underbrace{\|-f(\mathcal V_{\boldsymbol\mu, \boldsymbol\xi})+\sum_{j=1}^k \beta_j \mathcal W_{\mu_j, \xi_j}^2+f(z_{\lambda_0})\|_{L^{\frac 32}}}_{\|(I)\|_{L^{\frac 32}}}
\\&+\underbrace{\|-f(z_{\lambda_0})+f(u_{\lambda_0})+\ve^2{\rm sgn}(u_{\lambda_0}(x))v_{\lambda_0}^2+\ve f'(u_{\lambda_0})v_{\lambda_0}+\ve^2 f'(u_{\lambda_0})w_{\lambda_0}\|_{L^{\frac32}}}_{\|(II)\|_{L^{\frac32}}}\\
&+\|\sum_{j=1}^k \beta_j (\mathcal U^2_{\mu_j, \xi_j}-\mathcal W_{\mu_j, \xi_j}^2)\|_{L^{\frac 32}}+(\lambda_0+|\ve|)\sum_{j=1}^k \|\mathcal W_{\mu_j, \xi_j}\|_{L^{\frac 32}}+\underbrace{|\ve|^3\|w_{\lambda_0}\|_{L^{\frac 32}}}_{:=\mathcal O(|\ve|^3)}.\end{aligned}$$

It follows from \eqref{talenti}, \eqref{expW} and \eqref{muj} that
$$\|\mathcal W_{\mu_j, \xi_j}\|_{L^{\frac 32}}\lesssim \|\mathcal U_{\mu_j, \xi_j}\|_{L^{\frac 32}}\lesssim \bar\mu^2|\ln\bar\mu|^{\frac 23}$$ and
$$\begin{aligned}\|\mathcal W^2_{\mu_j, \xi_j}-\mathcal U^2_{\mu_j, \xi_j}\|_{L^{\frac 32}}&\lesssim \left(\int_\Omega \underbrace{|\mathcal W_{\mu_j, \xi_j}-\mathcal U_{\mu_j, \xi_j}|^{\frac 32}}_{:=\mathcal O(\bar\mu^3)}|\mathcal W_{\mu_j, \xi_j}+\mathcal U_{\mu_j, \xi_j}|^{\frac 3 2}\,dx\right)^{\frac 23}\\ &\lesssim \bar\mu^2 \|\mathcal U_{\mu_j, \xi_j}\|_{L^{\frac 32}}=o\left(\bar\mu^2|\ln\bar\mu|^{\frac 23}\right).\end{aligned}$$
To estimate $\|(I)\|_{L^{\frac 32}}$, we need to divide the whole domain $\Omega$ into two parts. The first part is in a small neighborhood of every point $\xi_h$ to be $B_{\sqrt \mu_h}(\xi_h)$. In every $B_{\sqrt \xi_h} (\xi_h)$, we rewrite
\begin{align*}
  f\left(\mathcal{V}_{\pmb{\mu,\xi}} \right) -\sum_{j=1}^{k}\beta_j\mathcal{W}_{\mu_j,\xi_j}^2-f(z_{\lambda_0})
=& f\left(\beta_h\Wh +\sum_{j\neq h}\beta_j \Wj +z_{\lambda_0}\right)\\
& -\beta_h \Wh^2 -\sum_{j\neq h}\beta_j \Wj^2-f(z_{\lambda_0}),
\end{align*}
which, together with Taylor expansion, implies that
$$\begin{aligned}&\left|
f\left(\beta_h\Wh +\sum_{j\neq h}\beta_j \Wj +z_{\lambda_0}\right) -\beta_h \Wh^2 -\sum_{j\neq h}\beta_j \Wj^2-f(z_{\lambda_0})\right|\\
\lesssim\quad &\mathcal{W}_{\mu_h,\xi_h}\left(\sum_{j\neq h}\mathcal{W}_{\mu_j,\xi_j}+|z_{\lambda_0}|\right)+\sum_{j\neq h}\mathcal{W}_{\mu_j,\xi_j}^2+|f(z_{\lambda_0})|.
\end{aligned}$$
In the remaining region $\Omega\setminus\bigcup_{h=1}^k B_{\sqrt\mu_h}(\xi_h)$, we rewrite
$$\begin{aligned}
 f\left(\mathcal{V}_{\pmb{\mu,\xi}}\right)-f(z_{\lambda_0})-\sum_{j=1}^{k}\beta_j\mathcal{W}_{\mu_j,\xi_j}^2
=  f\left(z_{\lambda_0}+\sum_{j=1}^{k}\beta_j\Wj\right)-f(z_{\lambda_0})-\sum_{j=1}^{k}\beta_j\mathcal{W}_{\mu_j,\xi_j}^2,
\end{aligned}$$
which, together with Taylor expansion again, implies that
$$\begin{aligned}
\left|  f\left(z_{\lambda_0}+\sum_{j=1}^{k}\beta_j\Wj\right)-f(z_{\lambda_0})-\sum_{j=1}^{k}\beta_j\mathcal{W}_{\mu_j,\xi_j}^2\right|
\lesssim \sum_{j=1}^{k} \left(\mathcal{W}_{\mu_j,\xi_j}+\mathcal{W}_{\mu_j,\xi_j}^2\right).
\end{aligned}$$
Now, by direct computations with \eqref{talenti}, \eqref{expW}, \eqref{muj} and \eqref{xij}, we get that
$$\|(I)\|_{L^{\frac 32}}\lesssim\sum_{h=1}^{k}\|(I)\|_{L^{\frac 32},B_{\sqrt\mu_h}(\xi_h)}+\|(I)\|_{L^{\frac 32},\Omega\setminus\bigcup_{h=1}^k B_{\sqrt\mu_h}(\xi_h)}\lesssim \bar\mu^2|\ln\bar\mu|^{\frac 23}. $$
For what concerning $\|(II)\|_{L^{\frac32}}$, we notice that
$$
 f(z_{\lambda_0})-f(u_{\lambda_0})-\ve f'(u_{\lambda_0}) v_{\lambda_0} -\ve^2f'(u_{\lambda_0})w_{\lambda_0}-\text{sgn}(u_{\lambda_0})\ve^2 v_{\lambda_0}^2=\mathcal{O}(|\ve|^3),
$$
Then $\|(II)\|_{L^{\frac 32}}\lesssim |\ve|^3$ and the thesis follows.
\end{proof}
Then, we deal with the invertibility of the linearized operator $\mathcal L_{\boldsymbol\mu, \boldsymbol\xi}$ defined in \eqref{lin}. Since the proof is almost standard, we omit it (see \cite{V} Lemma 2.4 or \cite{RV} Lemma 4.2).
\begin{lemma}\label{invertibility}
For any $(\xi_j)_{j=1, \ldots, k}$ satisfying \eqref{xij} and any $(\mu_j)_{j=1,\ldots, k}$ satisfying \eqref{muj} there exists $C>0$ such that  for any $|\ve|$ sufficiently small it holds
\begin{equation*}
\|\mathcal L_{\boldsymbol\mu, \boldsymbol\xi}(\varphi_\lambda)\|\geq C \|\varphi_\lambda\|,\quad \forall\,\,\varphi_\lambda\in \mathcal K_{\boldsymbol\mu, \boldsymbol\xi}^\bot.\end{equation*}
In particular, $\mathcal L_{\boldsymbol\mu, \boldsymbol\xi}$ is invertible with inverse which is uniformly bounded.
\end{lemma}

With Lemmas~\ref{stimaerrore} and \ref{invertibility} in hands, we are in position now to find a solution of the equation \eqref{ort} whose proof relies on a standard contraction mapping argument (see for example \cite{MP}, Proposition 1.8 and \cite{MPVe}, Proposition 2.1).
\begin{proposition}\label{existencevarphi}
For any $(\xi_j)_{j=1, \ldots, k}$ satisfying \eqref{xij} and any $(\mu_j)_{j=1,\ldots, k}$ satisfying \eqref{muj} there exists $C>0$ such that for any $|\ve|$ sufficiently small there exists a unique $\varphi_\lambda\in \mathcal K_{\boldsymbol\mu, \boldsymbol\xi}^\bot$ solution of \eqref{ort} which is continuously differentiable with respect to $\mu_j$ and $\xi_j$. Moreover
\be\label{estimatevarphi}
\|\varphi_\lambda\|\leq C \left(\bar\mu^2|\ln\bar\mu|^{\frac 23}+|\ve|^3\right).\ee

\end{proposition}
\section{ The finite dimensional problem}
Let now $\varphi_\lambda$ be the solution found in Proposition \ref{existencevarphi}. To solve \eqref{bif}, we have to find the parameters $\pmb\mu(\ve)$ and $ \pmb\xi(\ve)$ satisfying \eqref{muj} and \eqref{xij}  such that $\varphi_\lambda$ satisfies the equation \eqref{bif}. It is well known that the reduced problem has a variational structure, that is, the solutions of \eqref{pb} are critical points of the $C^2-$ functional $J_\lambda:H^1_0(\Omega) \to \mathbb R$ defined as \begin{equation*}\label{functional}
 J_\lambda(u):=\frac 12 \int_\Omega |\nabla u|^2\, dx-\frac{\lambda}{2}\int_\Omega u^2\, dx -\int_\Omega F(u)\, dx\end{equation*} where $F(s):=\int_0^s f(t)\, dt$.  Now let $\widetilde J_\lambda: \mathbb{R}_+^k\times \mathbb{R}^6\to \mathbb{R}$ be the reduced energy  given by
\begin{equation*}\label{re}
 \widetilde J_\lambda({\boldsymbol\mu, \boldsymbol\xi})=J_\lambda(\mathcal V_{\boldsymbol\mu, \boldsymbol\xi}+\varphi_\lambda).\end{equation*}
We shall give explicit asymptotic expressions for the reduced energy.
\begin{proposition}\label{reduced}
\begin{itemize}
\item[(i)] If $({\boldsymbol\mu, \boldsymbol\xi})$ satisfying \eqref{muj} and \eqref{xij} is a critical point of the reduced energy $\widetilde J_\lambda({\boldsymbol\mu, \boldsymbol\xi})$ then $\mathcal V_{\boldsymbol\mu, \boldsymbol\xi}+\varphi_\lambda$ is a solution of \eqref{pb}.
\item[(ii)] There exists $\ve_0>0$ such that for every $\ve\in (-\ve_0, \ve_0)$
\begin{equation}\label{re1}
\widetilde J_\lambda({\boldsymbol\mu, \boldsymbol\xi}):=\mathfrak c_0(\ve)-\sum_{j=1}^k\mathcal{E}_{\lambda,j}(\xi_j,\mu_j)+\mathcal{O}(|\ve|^2\bar{\mu}^2+|\ve|^6+\bar{\mu}^{\frac 72}+\bar{\mu}^3\left|\xi_j-\xi_{j,0}\right|^2)
\end{equation}
where
\begin{eqnarray}\label{re2}
\mathcal{E}_{\lambda,j}(\xi_j,\mu_j)=d_1\left(\frac{\lambda_0}{2}+\beta_ju_{\lambda_0}(\xi_j)\right)\mu_j^2+\ve d_1\left(\frac{1}{2}+\beta_jv_{\lambda_0}(\xi_j)\right)\mu_j^2
+\frac{11}{9}d_2\mu_j^3
\end{eqnarray}
with $d_1=\left\|\mathcal{U}\right\|_{L^2,\mathbb{R}^6}^2$ and $d_2=\alpha_6^{\frac 32}\omega_6\left|u_{\lambda_0}(\xi_{j,0})\right|^{\frac{3}{2}}$.
\end{itemize}
\end{proposition}
\begin{proof}
The proof of $(i)$ is quite standard so we omit it(see, for instance, Proposition 2.2 of \cite{MPVe}).
Now we show $(ii)$.  In order to simplify once more the notations we let $$\mathcal W_j:=\mathcal W_{\mu_j, \xi_j},\quad \mathcal U_j:=\mathcal U_{\mu_j, \xi_j}.$$
First, by using \eqref{estimatevarphi} and by making standard computations we get that
$$J_\lambda(\mathcal V_{\boldsymbol\mu, \boldsymbol\xi}+\varphi_\lambda)=J_\lambda(\mathcal V_{\boldsymbol\mu, \boldsymbol\xi})+\mathcal O\left(|\ve|^2\|\varphi_\lambda\|+\|\varphi_\lambda\|^2\right).$$
Now, we need to estimate the main term of $J_\lambda(\mathcal V_{\boldsymbol\mu, \boldsymbol\xi})$, i.e.,
$$\begin{aligned}
&J_\lambda(\mathcal V_{\boldsymbol\mu, \boldsymbol\xi})=  \underbrace{\frac12\|z_{\lambda_0}\|^2-\frac{\lambda_0+\ve}{2}\|z_{\lambda_0}\|^2
-\|z_{\lambda_0}\|_{L^3}^3}_{:=\mathfrak c_1(\ve)}+\underbrace{\frac 12 \sum_{j=1}^k\| \mathcal W_j\|^2 -\frac13\sum_{j=1}^k  \| \W_j\|_{L^3}^3}_{(A)}\\
&-\underbrace{\frac{\lambda_0}{2}\sum_{j=1}^k \| \W_j\|_{L^2}^2-\left\langle f\left(\sum_{j=1}^k \beta_j \W_j\right),u_{\lambda_0}\right\rangle_{L^2}}_{(B)}+\underbrace{\frac 12 \sum_{i\neq j}\beta_i\beta_j \left(\left\langle \W_i, \W_j\right\rangle-(\lambda_0+\ve) \left\langle\W_i,\W_j\right\rangle_{L^2}\right)}_{(C)}\\
&-\underbrace{\frac{\ve}{2}\sum_{j=1}^k \| \W_j\|_{L^2}^2-\ve \left\langle f\left(\sum_{j=1}^k\beta_j \W_j\right), v_{\lambda_0}\right\rangle_{L^2}}_{(E)}-\underbrace{\ve^2 \left\langle f\left(\sum_{j=1}^k\beta_j \W_j\right), w_{\lambda_0}\right\rangle_{L^2}}_{(F)}-\underbrace{\ve^3\sum_{j=1}^k\beta_j \left\langle \W_j, w_{\lambda_0}\right\rangle_{L^2}}_{(G)}\\
&-\underbrace{\left[\frac13\left(\left\|\mathcal{V}_{\pmb\mu,\pmb\xi}\right\|_{L^3}^3-\left\|z_{\lambda_0}\right\|_{L^3}^3-\sum_{j=1}^k \left\|\W_j\right\|_{L^3}^3\right)-\left\langle f\left(\sum_{j=1}^k\beta_j \W_j\right),z_{\lambda_0}\right\rangle_{L^2}-\left\langle f(z_{\lambda_0}),\sum_{j=1}^k\beta_j \W_j\right\rangle_{L^2}\right]}_{(H)}\\
&-\underbrace{\left\langle f(z_{\lambda_0})-f(u_{\lambda_0})-f'(u_{\lambda_0})\ve v_{\lambda_0}-f'(u_{\lambda_0})\ve^2 w_{\lambda_0}-\ve^2{\rm sgn}(u_{\lambda_0})v_{\lambda_0}^2,\sum_{j=1}^k\beta_j \W_j\right\rangle_{L^2}}_{(I)}.
\end{aligned}$$
where $z_{\lambda_0}$ is given by \eqref{fix}.
By using \eqref{expW} and \eqref{muj} we have that $\phi_j:=\W_j-\uu_j=\mathcal O(\bar\mu^2).$ Then it is easy to show that
$(A):=\frac k 6 \int_{\mathbb R^6}\mathcal U^3\, dx +\mathcal O(\bar\mu^4)=\mathfrak c_2 +\mathcal O(\bar\mu^4).$
It follows from \eqref{talenti}, \eqref{expW}, \eqref{muj} and \eqref{xij} that
$$
\begin{aligned}
(B)&=-\frac{\lambda_0}{2}\sum_{j=1}^k \left\|\uu_j+\phi_j\right\|_{L^2}^2-\sum_{j=1}^k \left\langle f(\beta_j\W_j), u_{\lambda_0}\right\rangle_{L^2}-\left\langle f\left(\sum_{j=1}^k \beta_j\W_j\right)-\sum_{j=1}^k  f(\beta_j\W_j),u_{\lambda_0}\right\rangle_{L^2}\\
&=-\sum_{j=1}^k\underbrace{\left(\frac{\lambda_0}{2} \left\|\uu_j\right\|^2-\beta_j\left\langle\uu_j^2, u_{\lambda_0}\right\rangle_{L^2}\right)}_{(B_1)}-\underbrace{\left\langle f\left(\sum_{j=1}^k \beta_j\W_j\right)-\sum_{j=1}^k  f(\beta_j\W_j),u_{\lambda_0}\right\rangle_{L^2}}_{(B_2)}+\mathcal O(\bar\mu^4)
\end{aligned}
$$
and
$$\begin{aligned}(B_1)&=-\frac{\lambda_0}{2}\mu_j^2\int_{\mathbb R^6}\uu^2\, dx -\beta_j \mu_j^2\int_{\mathbb R^6}\uu^2 u_{\lambda_0}(\mu_j y+\xi_j)\, dy+\mathcal O(\bar\mu^4)\\
&=-\frac{\lambda_0}{2}d_1\mu_j^2 -\beta_ju_{\lambda_0}(\xi_j) \mu_j^2\|\uu\|_{L^2,\bbr^6}^2+\mathcal O(\bar\mu^4)+\mathcal O(\bar\mu^3|\xi_j-\xi_{j, 0}|^2)\\
&=-\frac{\lambda_0}{2}d_1\mu_j^2 -\beta_ju_{\lambda_0}(\xi_j) d_1\mu_j^2+\mathcal O(\bar\mu^4)+\mathcal O(\bar\mu^3|\xi_j-\xi_{j, 0}|^2)\\\end{aligned}$$
while
$$\begin{aligned}(B_2)=&-\sum_{h=1}^k \left\langle f\left(\beta_h \W_h +\sum_{j\neq h} \beta_j \W_j\right) -f(\beta_h \W_h)-\sum_{j\neq h}f(\beta_j \W_j), u_{\lambda_0}\right\rangle_{L^2,B_{\rho}(\xi_h)}\\
&-\left\langle f\left(\sum_{j=1}^k \beta_j\W_j\right)-\sum_{j=1}^k  f(\beta_j\W_j),u_{\lambda_0}\right\rangle_{L^2,\Omega\setminus \bigcup_{h=1}^k B_{\rho}(\xi_h)}\\
\lesssim& \sum_{h=1}^k\sum_{j\neq k} \left\langle\W_h,\W_j\right\rangle_{L^2,B_{\rho}(\xi_h)}+\sum_{h=1}^k\sum_{j\neq k} \left\| \W_j\right\|_{L^2,B_{\rho}(\xi_h)}^2+\sum_{j=1}^k\left\|\W_j\right\|_{L^2,\Omega\setminus \bigcup_{h=1}^k B_{\rho}(\xi_h)}^2\\
\lesssim &\,\bar\mu^{\frac 7 2}.
\end{aligned}$$
Hence
$$(B):=-\mu_j^2\left(\frac{\lambda_0}{2} +\beta_ju_{\lambda_0}(\xi_j)\right) d_1+\mathcal O(\bar\mu^{\frac 72})+\mathcal O(\bar\mu^3|\xi_j-\xi_{j, 0}|^2).$$
In a very similar way, it is possible to show that
$$(E):=-\ve\mu_j^2 \left(\frac{1}{2}+\beta_jv_{\lambda_0}(\xi_j)\right) d_1 +\mathcal O(|\ve|\bar\mu^{\frac 72})+\mathcal O(|\ve|\bar\mu^3|\xi_j-\xi_{j, 0}|^2).$$
Since $-\Delta\W_i=\uu_i^3$ for all $1\le i\le k$, then it follows directly that $(C)=\mathcal O(\bar\mu^{\frac 72})$. Moreover
$$(F) \lesssim |\ve|^2\sum_{j=1}^k \|\W_j\|_{L^2}^2 =\mathcal O(|\ve|^2\bar\mu^2)$$ and
$$(G):=\mathcal O\left(|\ve|^3\into \W_j\,dx\right)=\mathcal O\left(|\ve|^3 \into\frac{\mu_j^2}{|x-\xi_j|^4}\,dx\right)=\mathcal O\left(|\ve|^3\bar\mu^2\right).$$
Furthermore $$(I)\lesssim |\ve|^2\sum_{j=1}^k\into \W_j\,dx\lesssim |\ve|^2\bar\mu^2.$$
For what concerning $(H)$, we divide $\Omega$ into the small neighborhood of the point $\xi_h$, say $B_{\sqrt[4]{\mu_h}}(\xi_h)$, for all $1\leq h\leq k$ and the remaining region of $\Omega$, that is  $\Omega\setminus \bigcup_{h=1}^k B_{\sqrt[4]{\mu_h}}(\xi_h) $. For simplicity, we let $B_h:=B_{\sqrt[4]{\mu_h}}(\xi_h)$. Then
$$\begin{aligned}(H)&=\int_{\bigcup_{h=1}^k B_h}\left[\dots\dots\right]\, dx +\int_{\Omega\setminus \bigcup_{h=1}^k B_h}\left[\dots\dots\right]\, dx \\
&=\sum_{h=1}^k \int_{B_h}\left[\dots\dots\right]\, dx +\int_{\Omega\setminus \bigcup_{h=1}^k B_h}\left[\dots\dots\right]\, dx. \end{aligned}$$
In every $B_h$, we rewrite
$$\begin{aligned}
&F\left(z_{\lambda_0}+\sum_{j=1}^k\beta_j \W_j\right)-F(z_{\lambda_0})-\sum_{j=1}^k F\left(\W_j\right)-f\left(\sum_{j=1}^k\beta_j \W_j\right)z_{\lambda_0}-f\left(z_{\lambda_0}\right)\sum_{j=1}^k\beta_j \W_j\\
=& \left[F\left(z_{\lambda_0}+\beta_h \W_h +\sum_{j\neq h}\W_j\right)-F(z_{\lambda_0}+\beta_h \W_h)-f(z_{\lambda_0}+\beta_h \W_h )\sum_{j\neq h}\W_j\right]\\
&+\left[F(z_{\lambda_0}+\beta_h \W_h)-F(z_{\lambda_0})- F(\W_h)-f(\beta_h \W_h)z_{\lambda_0}-f(z_{\lambda_0})\beta_h\W_h\right]\\
&+ \left[f\left(\sum_{j=1}^k\W_j\right)-f(\beta_h\W_h)\right]z_{\lambda_0}
+\left[f(z_{\lambda_0}+\beta_h\W_h)-f(z_{\lambda_0})\right]\sum_{j\neq h}\beta_j\W_j-\sum_{j\neq h}F(\W_j).
\end{aligned}$$
Now, we use  Taylor expansion and suitable computations with \eqref{talenti}, \eqref{expW}, \eqref{muj} and \eqref{xij} to obtain
$$\begin{aligned}\int_{B_h}\left[\dots\dots\right]\, dx =&\int_{B_h}\left[F(z_{\lambda_0}+\beta_h \W_h)-F(z_{\lambda_0})- F(\W_h)-f(\beta_h \W_h)z_{\lambda_0}-f(z_{\lambda_0})\beta_h\W_h\right]\, dx\\
&+\mathcal O\left(\int_{B_h} \sum_{j\neq h}F(\W_j)\, dx\right)+\mathcal O\left(\int_{B_h} \W_h\sum_{j\neq h}\W_j^2\, dx\right)\\
&+\mathcal O\left(\int_{B_h} \W_h\sum_{j\neq h}\W_j\, dx\right)+\mathcal O\left(\int_{B_h} \W_h^2\sum_{j\neq h}\W_j\, dx\right)\\
=&\int_{B_h}\left[F(z_{\lambda_0}+\beta_h \W_h)-F(z_{\lambda_0})- F(\W_h)-f(\beta_h \W_h)z_{\lambda_0}-f(z_{\lambda_0})\beta_h\W_h\right]\, dx\\
&+\mathcal O\left(\bar\mu^{\frac 7 2}\right).
\end{aligned}$$
In the remaining part $\Omega\setminus \bigcup_{h=1}^k B_h $, we rewrite
$$\begin{aligned}
&F\left(z_{\lambda_0}+\sum_{j=1}^k\beta_j \W_j\right)-F(z_{\lambda_0})-\sum_{j=1}^k F\left(\W_j\right)-f\left(\sum_{j=1}^k\beta_j \W_j\right)z_{\lambda_0}-f\left(z_{\lambda_0}\right)\sum_{j=1}^k\beta_j \W_j\\
=&\left[F\left(z_{\lambda_0}+\sum_{j=1}^k\beta_j \W_j\right)-F(z_{\lambda_0})-f(z_{\lambda_0})\sum_{j=1}^k\beta_j \W_j\right]+\sum_{j=1}^k\beta_j F\left(\W_j\right)+f\left(\sum_{j=1}^k\beta_j \W_j\right)z_{\lambda_0}.
\end{aligned}$$
By Taylor expansion and direct computations with \eqref{talenti}, \eqref{expW}, \eqref{muj} and \eqref{xij} again, we get that
$$\begin{aligned}\int_{\Omega\setminus \bigcup_{h=1}^k B_h}\left[\dots\dots\right]\, dx&=\mathcal O\left(\sum_{j=1}^k\int_{\Omega\setminus \bigcup_{h=1}^k B_h}\W_j^3\,dx\right)+\mathcal O\left(\sum_{j=1}^k\int_{\Omega\setminus \bigcup_{h=1}^k B_h}\W_j^2\,dx\right)\\
&=\mathcal O\left(\bar\mu^{\frac 7 2}\right).\end{aligned}$$
Hence
$$(H)=\sum_{h=1}^{k}\underbrace{\int_{B_h}\left[F(z_{\lambda_0}+\beta_h \W_h)-F(z_{\lambda_0})- F(\W_h)-f(\beta_h \W_h)z_{\lambda_0}-f(z_{\lambda_0})\beta_h\W_h\right]\, dx}_{(H_h)}+\mathcal O\left(\bar\mu^{\frac 7 2}\right).$$
It remains to evaluate $(H_h)$ for all $1\leq h \leq k$. First we notice that
$$\left\{\begin{aligned} &\beta_h=-1\,\,&\mbox{and}\,\, z_{\lambda_0}\leq 0 \quad &\Longrightarrow \quad (H_h)=0,\\
&\beta_h=1\,\,&\mbox{and}\,\, z_{\lambda_0}\geq 0 \quad &\Longrightarrow \quad (H_h)=0.\end{aligned}\right.$$
Then, if $\beta_h=-1$ then \begin{equation}\label{-1}\begin{aligned}(H_h)&=\frac 13 \int_{\left\{z_{\lambda_0}(x)\geq \W_h\right\}\cap B_h}\left(-2\W_h^3+6z_{\lambda_0}\W_h^2\right)\, dx\\
&\quad+\frac 13 \int_{\left\{0\leq z_{\lambda_0}(x)<\W_h\right\}\cap B_h}\left(-2z_{\lambda_0}^3+6z_{\lambda_0}^2\W_h\right)\,dx\end{aligned}\end{equation}
while, if $\beta_h=1$ then \begin{equation*}\label{1}\begin{aligned}(H_h)&=\frac 13 \int_{\left\{z_{\lambda_0}(x)\leq -\W_h\right\}\cap B_h}\left(-2\W_h^3-6z_{\lambda_0}\W_h^2\right)\, dx \\&\quad+\frac 13 \int_{\left\{-\W_h< z_{\lambda_0}(x)\leq 0\right\}\cap B_h}\left(2z_{\lambda_0}^3+6z_{\lambda_0}^2\W_h\right)\,dx.\end{aligned}\end{equation*}
We remark that easily follows that
$$\int_{\left\{z_{\lambda_0}(x)\leq -\W_h\right\}\cap B_h}\left(-2\W_h^3-6z_{\lambda_0}\W_h^2\right)\, dx= \int_{\left\{z_{\lambda_0}(x)\geq \W_h\right\}\cap B_h}\left(-2\W_h^3+6z_{\lambda_0}\W_h^2\right)\, dx$$
while
$$\int_{\left\{-\W_h< z_{\lambda_0}(x)\leq 0\right\}\cap B_h}\left(2z_{\lambda_0}^3+6z_{\lambda_0}^2\W_h\right)\, dx=-\int_{\left\{0\leq z_{\lambda_0}(x)<\W_h\right\}\cap B_h}\left(-2z_{\lambda_0}^3+6z_{\lambda_0}^2\W_h\right)\, dx$$ and hence we can only evaluate
\eqref{-1}.\\
As done in \cite{PV2021}, it is possible to show that for $\ve$ sufficiently small and $\mu_j, \xi_j$ that satisfies \eqref{muj} and \eqref{xij} respectively we have that
\be\label{inclusion1}
B(\xi_h, R^1_h\sqrt\mu_h)\subset \left\{x\in\Omega\,:\, 0<z_{\lambda_0}(x)< \W_h\right\}\cap B_h \subset B(\xi_h, R^2_h\sqrt\mu_h)\ee
where \be\label{inclusion3}R^1_h, R^2_h=R_h^0+\mathcal O(|\ve|)\quad \mbox{with}\,\, R_h^0:=\left(\frac{\alpha_6}{|u_0(\xi_{h, 0})|}\right)^{\frac 14}.\ee Moreover
 \be\label{inclusion2}\begin{aligned}
B^c(\xi_h, R^2_h\sqrt\mu_h)\cap B_h &\subset \left\{x\in\Omega\,:\, z_{\lambda_0}(x)\geq \W_h\right\}\cap B_h\subset B^c(\xi_h, R^1_h\sqrt\mu_h).\end{aligned}\ee
It follows from \eqref{inclusion1} and \eqref{inclusion2} that
$$\begin{aligned}\int_{B^c(\xi_h, R^2_h\sqrt\mu_h)\cap B_h}\left(-2\W_h^3+6z_{\lambda_0}\W_h^2\right)\, dx &\leq \int_{\left\{x\in\Omega\,:\, z_{\lambda_0}(x)\geq \W_h\right\}\cap B_h}\left(-2\W_h^3+6z_{\lambda_0}\W_h^2\right)\, dx\\& \leq \int_{B^c(\xi_h, R^1_h\sqrt\mu_h)}\left(-2\W_h^3+6z_{\lambda_0}\W_h^2\right)\, dx. \end{aligned}$$
Now, as done in \cite{PV2021}, let $R_h$ denote either $R_h^1$ or $R_h^2$. Then, by \eqref{inclusion3} we get
$$\begin{aligned}&\frac 13\int_{B^c(\xi_h, R^2_h\sqrt\mu_h)\cap B_h}\left(-2\W_h^3+6z_{\lambda_0}\W_h^2\right)\, dx \\
=&-\frac 2 3\int_{R_h\sqrt{\mu_h}\leq |x-\xi_h|\leq \mu_h^{\frac 1 4}}\uu_h^3\, dx+2\int_{R_h\sqrt{\mu_h}\leq |x-\xi_h|\leq \mu_h^{\frac 1 4}}u_{\lambda_0}(x)\uu_h^2\, dx+\mathcal O\left(|\ve|\mu_h^3\right)\\
=&-\frac 23\alpha_6^3 \int_{\frac{R_h}{\sqrt{\mu_h}}\leq |y|\leq \mu_h^{-\frac 3 4}}\frac{1}{(1+|y|^2)^6}\, dy+2u_{\lambda_0}(\xi_{h, 0})\alpha_6^2\mu_h^2 \int_{\frac{R_h}{\sqrt{\mu_h}}\leq |y|\leq \mu_h^{-\frac 3 4}}\frac{1}{(1+|y|^2)^4}\, dy+\mathcal O(|\ve| \mu_h^3)\\
=&\,\frac 8 9\mu_h^3 \omega_6 \alpha_6^{\frac 32} (u_{\lambda_0}(\xi_{h, 0}))^{\frac 32}+\mathcal O(|\ve| \mu_h^3),\end{aligned}$$
and by comparison
$$\frac 13 \int_{\left\{x\in\Omega\,:\, z_{\lambda_0}(x)\geq \W_h\right\}\cap B_h}\left(-2\W_h^3+6z_{\lambda_0}\W_h^2\right)\, dx=\frac 8 9\mu_h^3 \omega_6 \alpha_6^{\frac 32} (u_{\lambda_0}(\xi_{h, 0}))^{\frac 32}+\mathcal O(|\ve| \mu_h^3).$$
In a similar way, as in \cite{PV2021} it follows that $$\frac 13 \int_{\left\{0\leq z_{\lambda_0}(x)<\W_h\right\}\cap B_h}\left(-2z_{\lambda_0}^3+6z_{\lambda_0}^2\W_h\right)\,dx=\frac 13\mu_h^3 \omega_6 \alpha_6^{\frac 32} (u_{\lambda_0}(\xi_{h, 0}))^{\frac 32}+\mathcal O(|\ve| \mu_h^3).$$ The thesis follows.\end{proof}

\section{Solving The reduced problem: proof of the Theorems \ref{thm1} and \ref{thm2}}
It follows from \eqref{re1} and \eqref{re2} in Proposition~\ref{reduced} that the main term of the reduced energy may vanish. Therefore, to complete our constructions, we have to demonstrate the impact of topology of the domain $\Omega$ on the Brezis-Nirenberg problem \eqref{pb}, specifically by proving $u_{\lambda_0}$ is essentially non-degenerate. We begin by proving Theorem \ref{thm1} through a PDE argument.
\begin{proof}[Proof of Theorem \ref{thm1}]
Clearly, if $u_{\lambda}$ is essentially non-degenerate then by the definition, it is easy to see that $\Omega_{u_{\lambda},\frac{\lambda}{2}}\cup\Omega_{u_{\lambda},-\frac{\lambda}{2}}\not=\emptyset$.  Thus, to complete the proof, it remains to show $u_{\lambda}$ is essentially non-degenerate if $\Omega_{u_{\lambda},\frac{\lambda}{2}}\cup\Omega_{u_{\lambda},-\frac{\lambda}{2}}\not=\emptyset$.

Let
\begin{eqnarray*}
w_{\eta}(x)=\frac{(x-\eta)\cdot\nabla u_{\lambda}(x)}{2}+u_{\lambda}(x)-\lambda v_{\lambda}(x),
\end{eqnarray*}
where $\eta$ is an arbitrary fixed point in $\mathbb{R}^6$.  Since $u_{\lambda}$ and $v_{\lambda}$ are solutions of \eqref{pb} and \eqref{eqnPreWu0026}, respectively, we know that $w_{\eta}$ satisfies
\begin{eqnarray}\label{eqnPreWu0027}
\left\{
\aligned
&-\Delta w_{\eta}=\left(\lambda+2|u_{\lambda}|\right)w_{\eta},\quad\text{in }\Omega,\\
&w_{\eta}(x)=\frac{(x-\eta)\cdot\nabla u_{\lambda}(x)}{2}\quad\text{on }\partial\Omega.
\endaligned
\right.
\end{eqnarray}
Clearly, $\forall\xi\in\Omega_{u_{\lambda},\pm\frac{\lambda}{2}}$, we have $w_{\eta}(\xi)=u_{\lambda}(\xi)-\lambda v_{\lambda}(\xi)$ and $\quad\nabla w_{\eta}(\xi)=\frac{1}{2}\nabla^2u_{\lambda}(\xi)(\xi-\eta)-\lambda\nabla v_{\lambda}(\xi)$
for all $\eta\in\mathbb{R}^6$.  Suppose the contrary that $u_{\lambda}$ is not essentially non-degenerate.  Then
\begin{eqnarray*}
\Omega_{u_{\lambda},\frac{\lambda}{2}}\subset\left(\Omega_{v_{\lambda},\frac{1}{2}}\cap\Omega_{u_{\lambda},v_{\lambda},+}\right)\quad\text{and}\quad\Omega_{u_{\lambda},-\frac{\lambda}{2}}\subset\left(\Omega_{v_{\lambda},-\frac{1}{2}}\cap\Omega_{u_{\lambda},v_{\lambda},-}\right).
\end{eqnarray*}
It follows that $\nabla v_{\lambda}(\xi)\in\text{Ker}^{\perp}\left(\nabla^2u_{\lambda}(\xi)\right)$ and $v_{\lambda}(\xi)=\pm\frac{1}{2}$ for all $\xi\in\Omega_{u_{\lambda},\pm\frac{\lambda}{2}}$, which implies that $w_{\eta}(\xi)=0$ for all $\eta\in\mathbb{R}^6$ and all $\xi\in\Omega_{u_{\lambda},\pm\frac{\lambda}{2}}$.  Since $\nabla v_{\lambda}(\xi)\in\text{Ker}^{\perp}\left(\nabla^2u_{\lambda}(\xi)\right)$, we have
\begin{eqnarray*}
\nabla w_{\eta}(\xi)=\sum_{i=1}^{l}\left(\tau_i\frac{(\xi-\eta)\cdot e_i}{2}-\lambda\nabla v_{\lambda}(\xi)\cdot e_i\right)e_i,
\end{eqnarray*}
where $\tau_1<\tau_2<\cdots<\tau_l<0$ are the negative eigenvalues of the matrix $\nabla^2u_{\lambda}(\xi)$.  Note that $\eta$ is an arbitrary fixed point in $\mathbb{R}^6$, we can choose $\eta_0\in \mathbb{R}^6$ such that $\nabla w_{\eta_0}(\xi)=0$.  However, since $w_{\eta_0}$ solves \eqref{eqnPreWu0027}, by the classical regularity theory of elliptic equations, $w_{\eta_0}$ is smooth, which implies that the level set $\Omega_{w_{\eta_0},0}=\left\{x\in\Omega\mid w_{\eta_0}(x)=0\right\}$
is smooth.  Thus, by the strong maximum principle, $\left|\nabla w_{\eta_0}(y)\right|=\left|\frac{\partial w_{\eta_0}}{\partial {\bf n}}(y)\right|\not=0$
for all $y\in\Omega_{w_{\eta_0},0}$, where ${\bf n}$ is the unit normal vector of $\Omega_{w_{\eta_0},0}$.  It is a contradiction since $\xi\in\Omega_{w_{\eta_0},0}$ with  $\nabla w_{\eta_0}(\xi)=0$.  Thus, $u_{\lambda}$ is essentially non-degenerate.
\end{proof}

We are now in position to show the existence result Theorem \ref{thm2}.
\begin{proof}[Proof of Theorem \ref{thm2}]
Since $\Omega_{u_{\lambda_0},\frac{\lambda_0}{2}}\cup\Omega_{u_{\lambda_0},-\frac{\lambda_0}{2}}\not=\emptyset$, by Theorem~\ref{thm1}, $u_{\lambda_0}\not=0$ is an essentially non-degenerate solution of Brezis-Nirenberg problem~\eqref{pb}.  Thus, we can choose $\{\xi_{j,0}\}_{1\leq j\leq k}\subset\Omega_{u_{\lambda_0},*}$
and take
\begin{eqnarray*}
\beta_j=\left\{\aligned
&-1,\quad \xi_{j,0}\in\Omega_{u_{\lambda_0},\frac{\lambda_0}{2}}\setminus\left(\Omega_{v_{\lambda_0},\frac{1}{2}}\cap\Omega_{u_{\lambda_0},v_{\lambda_0},+}\right),\\
&1,\quad \xi_{j,0}\in\Omega_{u_{\lambda_0},\frac{\lambda_0}{2}}\setminus\left(\Omega_{v_{\lambda_0},-\frac{1}{2}}\cap\Omega_{u_{\lambda_0},v_{\lambda_0},-}\right).
\endaligned
\right.
\end{eqnarray*}
For $\xi_{j,0}\in\Omega_{u_{\lambda_0},\pm\frac{\lambda_0}{2}}\setminus\left(\Omega_{v_{\lambda_0},\pm\frac{1}{2}}\cap\Omega_{u_{\lambda_0},v_{\lambda_0},\pm}\right)$, either
\begin{enumerate}
\item[$(1)$]\quad $\xi_{j,0}\in\Omega_{u_{\lambda_0},\pm\frac{\lambda_0}{2}}\setminus\Omega_{v_{\lambda_0},\pm\frac{1}{2}}$ or
\item[$(2)$]\quad $\xi_{j,0}\in\left(\Omega_{u_{\lambda_0},\pm\frac{\lambda_0}{2}}\setminus\Omega_{u_{\lambda_0},v_{\lambda_0},\pm}\right)\cap\Omega_{v_{\lambda_0},\pm\frac{1}{2}}$.
\end{enumerate}
In the case~$(1)$, we assume that $\xi_j=\xi_{j,0}$, which leads to
\begin{eqnarray*}
\mathcal{E}_{\lambda,j}(\xi_j,\mu_j)=\ve d_1\left(\frac{1}{2}+\beta_jv_{\lambda_0}(\xi_j)\right)\mu_j^2+\frac{11}{9}d_2\mu_j^3+\mathcal{O}(|\ve|^2\bar{\mu}^2+|\ve|^6+\bar{\mu}^4\left|\ln\bar{\mu}\right|^{\frac{4}{3}}).
\end{eqnarray*}
Since $\frac{1}{2}+\beta_jv_{\lambda_0}(\xi_j)\not=0$, we can choose $\frac{\ve}{|\ve|}=-\text{sgn}\left(\frac{1}{2}+\beta_jv_{\lambda_0}(\xi_j)\right)$, $\bar{\mu}=|\ve|$ and $\tau_{j,0}$ to be the unique global minimum point of the function
\begin{eqnarray*}
P_j(\tau)=-d_1\left|\frac{1}{2}+\beta_jv_{\lambda_0}(\xi_j)\right|\tau^2+\frac{11}{9} d_2\tau^3
\end{eqnarray*}
as in \cite{PV2021}.  Thus, $\mathcal{E}_{\lambda,j}(\xi_j,\mu_j)$ has a unique critical point as $\ve\to0$.  In the case~$(2)$, we must have $\Omega_{u_{\lambda_0},\pm\frac{\lambda_0}{2}}\setminus\Omega_{u_{\lambda_0},v_{\lambda_0},\pm}\not=\emptyset$.  It follows that $l_{j}<6$, where $\{e_{j,1},e_{j,2},\cdots e_{j,l_{j}}\}\in \mathbb{S}^{5}$
is the basis of $\text{Ker}^{\perp}\left(\nabla^2u_{\lambda}(\xi_{j,0})\right)$.  We can assume that $\xi_j=\xi_{j,0}+\rho e_{j,l_{j}+1}$.  Then by the Taylor expansion,
\begin{eqnarray*}
\mathcal{E}_{\lambda,j}(\xi_j,\mu_j)&=&\ve d_1\left(\nabla v_{\lambda_0}(\xi_{j,0})\cdot e_{j,l_{j}+1}\right)\rho\mu_j^2+\frac{11}{9} d_2\mu_j^3\\
&&+\mathcal{O}(\rho^3\bar{\mu}^2+|\ve|^2\bar{\mu}^2+|\ve|^6+\bar{\mu}^4\left|\ln\bar{\mu}\right|^{\frac{4}{3}}+\rho^2\bar{\mu}^3).
\end{eqnarray*}
We can choose $\frac{\ve}{|\ve|}=-\text{sgn}\left(\nabla v_{\lambda_0}(\xi_{j,0})\cdot e_{j,l_{j}+1}\right)$, $\rho=|\ve|^{s}$ with $\frac{1}{2}<s<1$, $\bar{\mu}=|\ve|^{1+s}$ and $\tau_{j,0}$ to be the unique global minimum point of the function
\begin{eqnarray*}
Q_j(\tau)=-d_1\left|\nabla v_{\lambda_0}(\xi_{j,0})\cdot e_{j,l_{j}+1}\right|\tau^2+\frac{11}{9} d_2\tau^3.
\end{eqnarray*}
Thus, $\mathcal{E}_{\lambda,j}(\xi_j,\mu_j)$ also has a unique critical point as $\ve\to0$.  It follows from standard arguments that $u_{\lambda}$ given by \eqref{eqnPreWu0001} is a solution of Brezis-Nirenberg problem~\eqref{pb} as $\lambda\to\lambda_0$, where $\Omega$ is a non-degenerate domain and $\lambda_0$ is a special value such that the Brezis-Nirenberg problem~\eqref{pb} in the non-degenerate domain $\Omega$ has a solution $u_{\lambda_0}$ such that $\Omega_{u_{\lambda_0},\frac{\lambda_0}{2}}\cup\Omega_{u_{\lambda_0},-\frac{\lambda_0}{2}}\not=\emptyset$.
\end{proof}

\end{document}